\documentclass[11pt]{amsart}
\usepackage{graphicx}
\usepackage{graphics, epsfig}
\newtheorem{thm}{Theorem}[section]
\newtheorem{cor}[thm]{Corollary}

\newtheorem{lem}[thm]{Lemma}
\newtheorem{prop}[thm]{Proposition}
\theoremstyle{definition}
\newtheorem{defn}[thm]{Definition}
\theoremstyle{remark}

\newtheorem{example}[thm]{Example}
\numberwithin{equation}{section}

\begin{document}

\title[Conditions for recurrence and transience]{Conditions for recurrence and transience for one family of random walks}%
\author{Vyacheslav M. Abramov}%
\address{24 Sagan Drive, Cranbourne North, Vic-3977, Australia}%

\email{vabramov126@gmail.com}%

\begin{abstract}
A parametric family of two-dimensional random walks $\mathbf{S}_t(\mathfrak{a})$ $=\big(S_t^{(1)}(\mathfrak{a}), S_t^{(2)}(\mathfrak{a})\big)$ in the main quarter plane, where $\mathfrak{a}$ is an infinite-dimensional vector-valued parameter, is studied.
The components $S_t^{(1)}(\frak{a})$ and $S_t^{(2)}(\mathfrak{a})$ are assumed to be correlated in the specified way that is defined exactly in the paper.
 We derive the conditions on $\frak{a}$, under which a random walk $\mathbf{S}_t(\mathfrak{a})$ is recurrent.

\end{abstract}
\maketitle

\section{Introduction and formulation of the main results}\label{S1}

\subsection{Definition of $\psi$-random walk}
The paper studies special classes of random walks in a quarter plane. The theory of random walks in a quarter plane is a well-established area having numerous applications (e.g. \cite{AFM, FIM, FMM, Raschel}). In the sequel, the random walks in the main quarter plane will be called \textit{reflected} random walks. Sometimes the word \textit{`reflected'} will be omitted.

A parametric family of random walks $\mathbf{S}_t(\mathfrak{a})=\big(S_t^{(1)}(\mathfrak{a}), S_t^{(2)}(\mathfrak{a})\big)$ studied in the present paper is a set of two-dimensional reflected random walks
depending on parameter $\mathfrak{a}\in\mathcal{A}$. It will be further denoted by $\{\mathbf{S}_t,\mathcal{A}\}$. The elements $\mathfrak{a}\in\mathcal{A}$ are infinite sequences of real values, and a special element $\mathfrak{o}\in\mathcal{A}$, which is the infinite sequence of zeros, is associated with the simple reflected two-dimensional random walk denoted by $\mathbf{S}_t(\mathfrak{o})$.

A random walk $\mathbf{S}_t(\mathfrak{a})$ starting at zero is recursively defined as follows. Let $\mathbf{e}_t(\mathfrak{a})=\big(e_t^{(1)}(\mathfrak{a}), e_t^{(2)}(\mathfrak{a})\big)$ be the vector taking the values $\{\pm\mathbf{1}_i, i=1,2\}$ with probabilities depending on $\mathfrak{a}\in\mathcal{A}$ and the state of the random walk at time $t$ ($t=0,1,\ldots$), where the vector $\mathbf{1}_i$ is the two-dimensional vector, the $i$th component of which is 1 and the remaining component is 0. (The further details about the vector $\mathbf{e}_t(\mathfrak{a})$ are given later.) Then,
\begin{eqnarray}
  \mathbf{S}_0(\mathfrak{a}) &=& \mathbf{0},\label{eq.1.24} \\
  \mathbf{S}_t(\mathfrak{a}) &=& \mathbf{S}_{t-1}(\mathfrak{a})+\mathbf{r}_t(\mathfrak{a}),\label{eq.1.25}
\end{eqnarray}
where $\mathbf{0}=(0,0)$,
$$
\mathbf{r}_t(\mathfrak{a})=\begin{cases}\mathbf{e}_t(\mathfrak{a}), &\text{if} \ {S}_{t-1}^{(i)}(\mathfrak{a})+e_t^{(i)}(\mathfrak{a})\geq0 \ \text{for both} \ i=1,2,\\
-\mathbf{e}_t(\mathfrak{a}), &\text{if} \ {S}_{t-1}^{(i)}(\mathfrak{a})+e_t^{(i)}(\mathfrak{a})=-1 \ \text{for one of} \ i=1,2.
\end{cases}
$$

Non-homogeneous random walks and their classification have been studied in book \cite{M_et_all}, where Lyapunov function methods for near-critical stochastic systems have been developed. Another new approach for classification of random walks through the concepts of conservative and semi-conservative random walks has been provided in \cite{A}.
In the present paper, we develop the methods \cite{A}, and the main contribution of the present paper is a closed form explicit asymptotic formula for the
parameter that enables us to classify whether a random walk $\mathbf{S}_t(\mathfrak{a})$ is transient or recurrent for a
quite general family of so-called \textit{nearest-neighborhood} random walks $\{\mathbf{S}_t, \mathcal{A}\}$.

Unfortunately, the families of conservative and semi-conservative random walks (in the sense of \cite{A}) are relatively narrow and seems cannot be used for the models considered in this paper. Therefore, we provide another classification of random walks that is closely related to one given in \cite{A}.
For any vector $\mathbf{n}=(n^{(1)}, n^{(2)},\ldots, n^{(d)})$, $\|\mathbf{n}\|$ denotes its $l_1$-norm, i.e.
$\|\mathbf{n}\|=|n^{(1)}|+|n^{(2)}|+\ldots+|n^{(d)}|$. Since in our case, the vectors are assumed to be two-dimensional with nonnegative components, the notation reduces to $\|\mathbf{n}\|=n^{(1)}+n^{(2)}$.

For the random walks defined by \eqref{eq.1.24}, \eqref{eq.1.25} denote
\begin{eqnarray}
P_t(\mathfrak{a}; n+1|n)&=&\mathsf{P}\{\|\mathbf{S}_{t+1}(\mathfrak{a})\|=n+1~\big|~\|\mathbf{S}_{t}(\mathfrak{a})\|=n\},\label{eq.1.26}\\
P_t(\mathfrak{a}; n-1|n)&=&\mathsf{P}\{\|\mathbf{S}_{t+1}(\mathfrak{a})\|=n-1~\big|~\|\mathbf{S}_{t}(\mathfrak{a})\|=n\}.\label{eq.1.27}
\end{eqnarray}
The condition $\{\|\mathbf{S}_{t}(\mathfrak{a})\|=n\}$ is not empty if and only if $t-n$ is even.
If $\{\|\mathbf{S}_{t}(\mathfrak{a})\|=n\}=\emptyset$, then the right-hand sides of \eqref{eq.1.26} and \eqref{eq.1.27} have the form of $\mathsf{P}\{\emptyset|\emptyset\}$, and relations \eqref{eq.1.26} and \eqref{eq.1.27} are undefined.

To extend \eqref{eq.1.26} and \eqref{eq.1.27} for all nonnegative combinations $t$ and $n$, the following setting is needed: if $\{\|\mathbf{S}_{t}(\mathfrak{a})\|=n\}=\emptyset$, we are to set $P_t(\mathfrak{a}; n+1|n)=P_{t-1}(\mathfrak{a}; n+1|n)$ and $P_t(\mathfrak{a}; n-1|n)=P_{t-1}(\mathfrak{a}; n-1|n)$, $t\geq1$.
Then the  random walk given by \eqref{eq.1.24}, \eqref{eq.1.25} can be redefined as follows. Instead of \eqref{eq.1.24} and \eqref{eq.1.25} one can consider the infinite series of identical
random walks given on the same probability space, assuming that the $N$th one
of them starts at time $t=-N$, i.e. $\mathbf{S}_{-N}^{(N)}(\mathfrak{a})=\mathbf{0}$, where the superscript
$(N)$ indicates the series number. This construction enables us to extend the
process to the whole time axis. For further
details see \cite{Kolmogorov}. Then, instead of \eqref{eq.1.24} and \eqref{eq.1.25} one can consider the process
defined recursively by
\begin{equation}\label{eq.1.2}
\mathbf{S}_t(\mathfrak{a}) = \mathbf{S}_{t-1}(\mathfrak{a})+\mathbf{r}_t(\mathfrak{a}),
\end{equation}
for all $t=\ldots,-1,0,1,\ldots$, and the limit
$$
\lim_{n\to\infty}\left(\frac{P_0(\mathfrak{a}; n+1|n)}{P_0(\mathfrak{a}; n-1|n)}\right)^n
$$
can be assumed to exist.

\begin{defn}\label{D1}
A family of random walks $\{\mathbf{S}_t, \mathcal{A}\}$, $t=\ldots,-1,0,1\ldots$, is said to have index $\psi$, if for all $\mathfrak{a}\in\mathcal{A}$
\begin{equation}\label{eq.1.3.2}
\lim_{n\to\infty}\left(\frac{P_0(\mathfrak{a}; n+1|n)}{P_0(\mathfrak{a}; n-1|n)}\right)^n
=\mathrm{e}^\psi.
\end{equation}
For simplicity, a family of random walks having index $\psi$ will be called $\psi$-random walks.
\end{defn}

Following this definition, the family of $d$-dimensional symmetric random walks \cite{A} is on the one hand $(\mathcal{A}, d)$-conservative and, on the other hand, represents $(d-1)$-random walks (see Theorem 2.1 and relation (4.6) of Lemma 4.2 in \cite{A}).

\subsection{Description of the model and formulation of main results}\label{S1.2}

In the present paper, we study the following new family of two-dimensional random walks. The set $\mathcal{A}$ is the set of infinite sequences, the elements $\mathfrak{a}$ of which are specified as indexed sequences $\{\alpha_{\mathbf{n}}\}$, where $\mathbf{n}=\{n^{(1)},n^{(2)}\}$ is the set of ordered pairs of nonnegative integers, $n^{(1)}\leq n^{(2)}$. So, $\alpha_{\mathbf{n}}=\alpha_{(n^{(1)},n^{(2)})}$, and in the rest of the paper each of the notation $\alpha_{\mathbf{n}}$ or $\alpha_{(n^{(1)},n^{(2)})}$ can be used interchangeably.

For the following specification of a random walk $\mathbf{S}_t(\mathfrak{a})=\big(S_t^{(1)}(\mathfrak{a}), S_t^{(2)}(\mathfrak{a})\big)$ denote:
\begin{eqnarray*}
  I_t(\mathfrak{a}) &=& \min\left\{S_t^{(1)}(\mathfrak{a}), S_t^{(2)}(\mathfrak{a})\right\}, \\
  J_t(\mathfrak{a}) &=& \max\left\{S_t^{(1)}(\mathfrak{a}), S_t^{(2)}(\mathfrak{a})\right\}, \\
  N_t(\mathfrak{a}) &=& \|\mathbf{S}_t(\mathfrak{a})\|=I_t(\mathfrak{a})+J_t(\mathfrak{a}).
\end{eqnarray*}
For the sake of simplicity, we use the notation $J_t=J_t(\mathfrak{a})$, $I_t=I_t(\mathfrak{a})$ and $N_t=N_t(\mathfrak{a})$ omitting the argument $\mathfrak{a}$ in all places where it is possible to do without loss of meaning. Then, the definition of the random walk follows from the following conditions:

\begin{eqnarray}
  \mathsf{P}\{J_{t+1}=J_t+1~|~I_t>0, I_t<J_t\} &=& \frac{1}{4}+\frac{\alpha_{(I_t,J_t)}}{N_t},\label{eq.1.4} \\
  \mathsf{P}\{J_{t+1}=J_t-1~|~I_t>0, I_t<J_t\} &=& \frac{1}{4}-\frac{\alpha_{(I_t,J_t)}}{N_t},\label{eq.1.5} \\
  \mathsf{P}\{I_{t+1}=I_t+1~|~I_t>0, I_t<J_t\} &=& \frac{1}{4}-\frac{\alpha_{(I_t,J_t)}}{N_t},\label{eq.1.6} \\
  \mathsf{P}\{I_{t+1}=I_t-1~|~I_t>0, I_t<J_t\} &=& \frac{1}{4}+\frac{\alpha_{(I_t,J_t)}}{N_t}.\label{eq.1.7}
\end{eqnarray}
\begin{equation}\label{eq.1.8}
\begin{aligned}
  &\mathsf{P}\{J_{t+1}=J_t+1~|~I_t>0, I_t=J_t\}\\
  &=\mathsf{P}\{I_{t+1}=I_t-1~|~I_t>0, I_t=J_t\}\\
  &=\frac{1}{2}.
\end{aligned}
\end{equation}
%


\begin{eqnarray}
    \mathsf{P}\{J_{t+1}=J_t+1~|~I_t=0, J_t>0\} &=& \frac{1}{4},\label{eq.1.9} \\
    \mathsf{P}\{J_{t+1}=J_t-1~|~I_t=0, J_t>0\} &=& \frac{1}{4},\label{eq.1.10}\\
    \mathsf{P}\{I_{t+1}=1~|~I_t=0, J_t>0\}&=&\frac{1}{2}.\label{eq.1.11}
\end{eqnarray}

\begin{eqnarray}
\mathsf{P}\{J_{t+1}=1~|~I_t=0, J_t=0\}&=&1,\label{eq.1.12}\\
\mathsf{P}\{I_{t+1}=0~|~I_t=0, J_t=0\}&=&1.\label{eq.1.13}
\end{eqnarray}
To be in an agreement with \eqref{eq.1.4} -- \eqref{eq.1.7}, the values $\alpha_{\mathbf{n}}$ are assumed to obey the inequality
\begin{equation}\label{eq.1.20}
|\alpha_{\mathbf{n}}|<\min\left\{C,\frac{1}{4}\|\mathbf{n}\|\right\},
\end{equation}
where $C$ is some positive constant. Note, that the quantities $\alpha_{(0,n)}$ and $\alpha_{(i,i)}$ for $n=0,1,\ldots$ and $i=0,1,\ldots$ are not used in the definition of the family of random walks and, hence, are equated to zero.

\begin{figure}
\includegraphics[width=10cm, height=14cm]{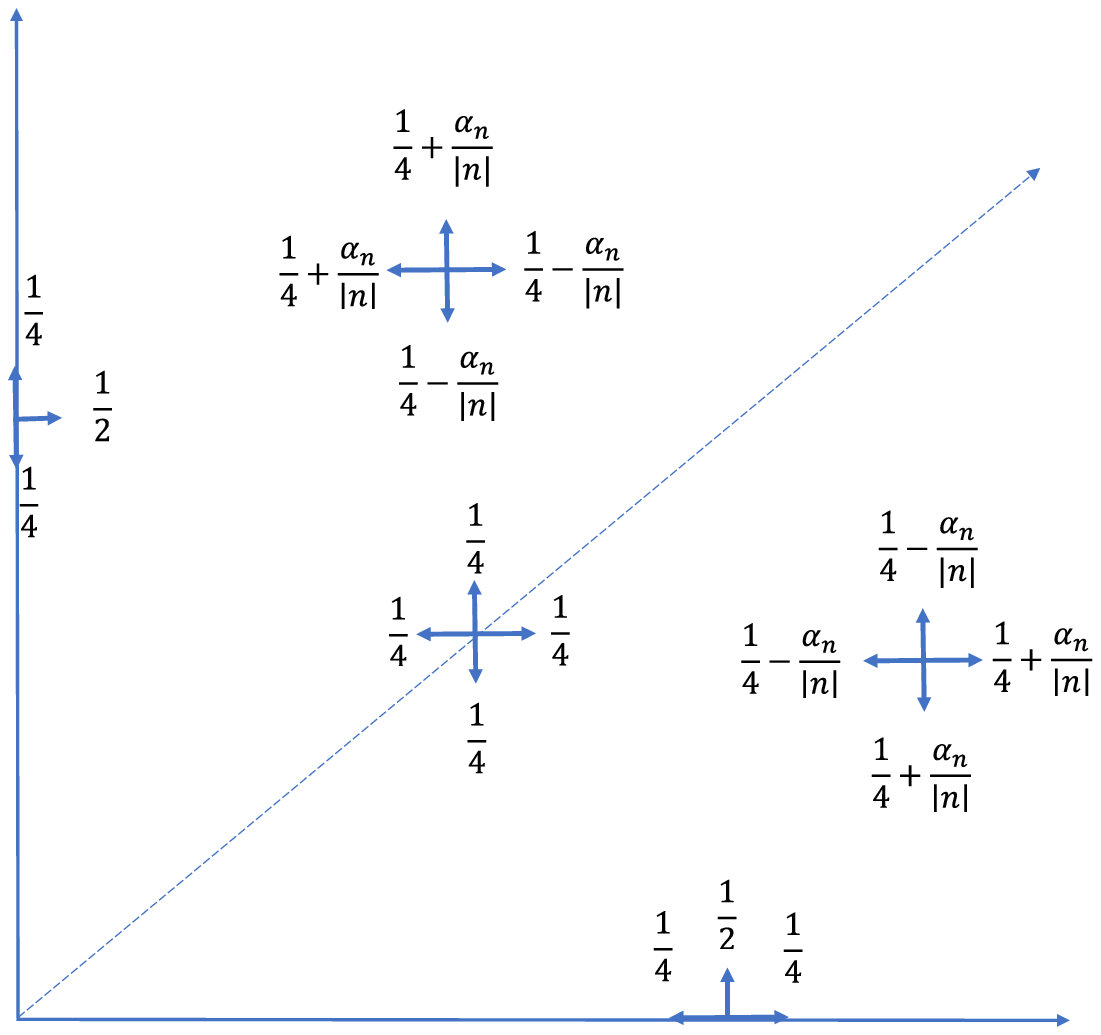}
\caption{Transition probabilities for a random walk in the main quarter plane.}
\end{figure}

Transition probabilities of a typical reflected random walk are illustrated on Figure 1. The $l_1$-norm of the vector $\mathbf{n}$ there is shown as $|n|$, and $\alpha_{\mathbf{n}}$ as $\alpha_n$. As it is seen from Figure 1, the transition probabilities are reflected symmetrically by the bisectrix of the main quarter plane. Because of this symmetry, the model can be further reduced to the random walk in a wedge (see Figure 2).

\begin{figure}
\includegraphics[width=10cm, height=14cm]{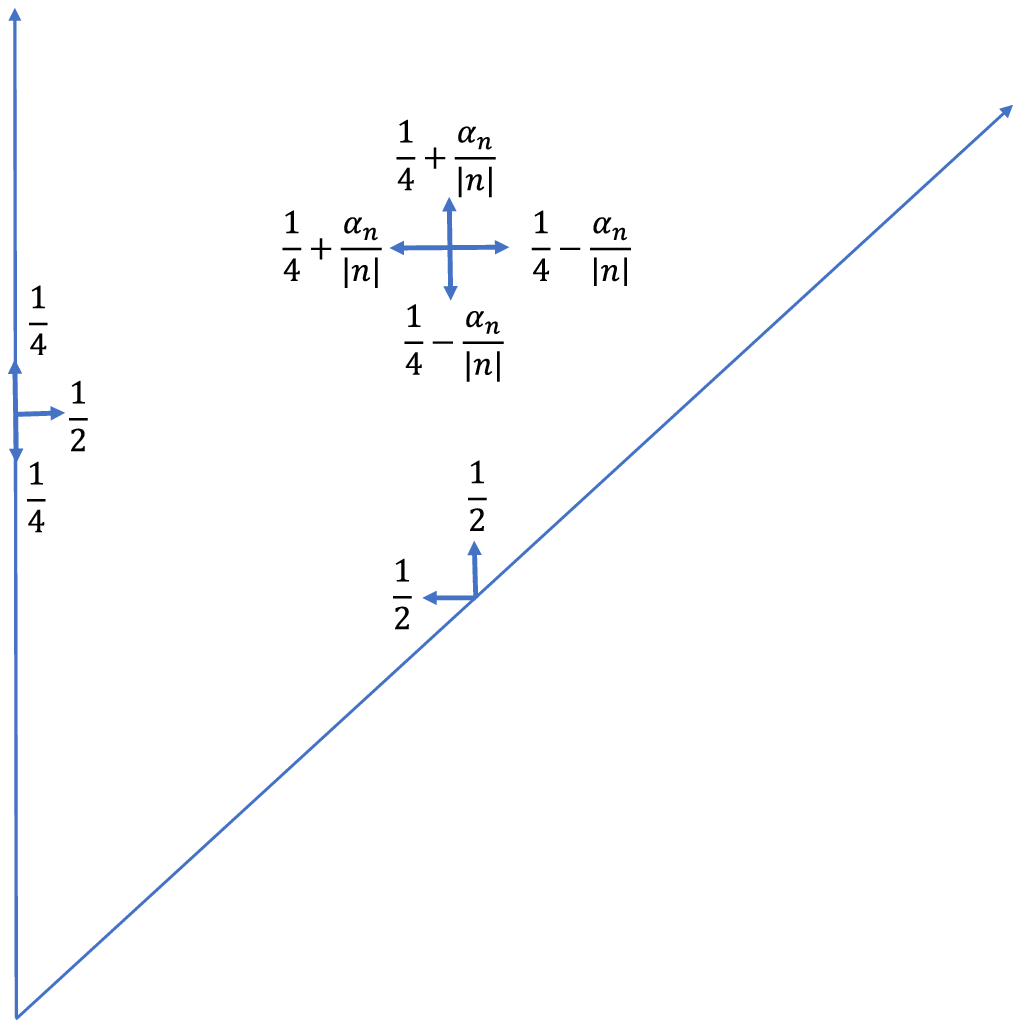}
\caption{Transition probabilities for a random walk in a wedge.}
\end{figure}

Assume that for all $\mathbf{n}$ with $n^{(1)}\geq2$ satisfying the inequality $\|\mathbf{n}\|>n_0$, where $n_0$ is some large value,
\begin{equation}\label{eq.1.21}
  |\alpha_{\mathbf{n}+\mathbf{1}}-\alpha_{\mathbf{n}}|\leq\gamma|\alpha_{\mathbf{n}}-\alpha_{\mathbf{n}-\mathbf{1}}|,
\end{equation}
where $\mathbf{1}=(1,1)$ and $0<\gamma<1$.

Assumption \eqref{eq.1.21} is technically used in Section \ref{S4.1}. Together with \eqref{eq.1.20} it implies the convergence of the coefficients $\alpha_{(n^{(1)},n^{(2)})}$ as $\|\mathbf{n}\|\to\infty$, in which the differences $n^{(2)}-n^{(1)}$ are kept fixed, to the limits with geometric rate. Namely, we denote
$$
\alpha^{*}_{n^{(2)}-n^{(1)}}=\lim_{k\to\infty}\alpha_{(n^{(1)}+k,n^{(2)}+k)}.
$$

\smallskip
In addition to \eqref{eq.1.20} and \eqref{eq.1.21}, assume that there exists the limit
\begin{equation}\label{eq.1.22}
\kappa(\mathfrak{a})=\lim_{k\to\infty}\frac{1}{k}\sum_{j=1}^{k-1}\prod_{i=1}^{j}\left(1+\frac{2\alpha^*_{2k-2i}}{k}+\frac{4\alpha^*_{2k-2i-1}}{k}+\frac{2\alpha^*_{2k-2i-2}}{k}\right).
\end{equation}
Notice that if the sequence of products
\begin{equation}\label{eq.1.23}
\prod_{i=1}^{k-1}\left(1+\frac{2\alpha^*_{2k-2i}}{k}+\frac{4\alpha^*_{2k-2i-1}}{k}+\frac{2\alpha^*_{2k-2i-2}}{k}\right)
\end{equation}
converges as $k\to\infty$, then $\kappa(\mathfrak{a})$ can be represented as
\begin{equation}\label{eq.1.50}
\kappa(\mathfrak{a})=\lim_{k\to\infty}\prod_{i=1}^{k-1}\left(1+\frac{2\alpha^*_{2k-2i}}{k}+\frac{4\alpha^*_{2k-2i-1}}{k}+\frac{2\alpha^*_{2k-2i-2}}{k}\right).
\end{equation}

The main result of the paper is the following theorem.
\begin{thm}\label{T1}
  If $\kappa(\mathfrak{a})\leq1$, then the random walk $\mathbf{S}_t(\mathfrak{a})$ is recurrent. Otherwise, the random walk $\mathbf{S}_t(\mathfrak{a})$ is transient.
\end{thm}

\subsection{Examples and intuition behind the main results}

\begin{example}\label{ex1}
The evident particular case of Theorem \ref{T1} is the case where $\mathfrak{a}=\{\alpha, \alpha,\ldots\}$ is the sequence of same value $\alpha$, which is similar to Example 4.1.9 of \cite{M_et_all} (based on the results of \cite{F, G}). In this case, if $\alpha>0$, then the random walk is transient, while in the opposite case, $\alpha\leq0$, the random walk is recurrent. Following Definition \ref{D1}, this random walk falls into the category of $\mathrm{e}^{8\alpha}$-random walks, since according to \eqref{eq.1.50}
$$
\kappa(\mathfrak{a})=\lim_{k\to\infty}\left(1+\frac{8\alpha}{k}\right)^{k-1}=\mathrm{e}^{8\alpha}.
$$
\end{example}

\noindent
\textit{Some intuition behind Theorem \ref{T1} and Example \ref{ex1}}. The centrally important issue explaining Theorem \ref{T1} is the fraction of time spent on the axis. Indeed, for large $n$ we have
\begin{equation*}\label{eq.1.30}
\frac{P_0(\mathfrak{a}; n+1|n)}{P_0(\mathfrak{a}; n-1|n)}=\frac{\lambda_n(\mathfrak{a})}{\mu_n(\mathfrak{a})},
\end{equation*}
where
\begin{equation*}\label{eq.1.31}
\begin{aligned}
\lambda_n(\mathfrak{a})\asymp&\frac{1}{2}\mathsf{P}\{\mathbf{S}_1(\mathfrak{a})~\text{in interior}~|~\|\mathbf{S}_{0}(\mathfrak{a})\|=n\}\\
&+\frac{3}{4}\mathsf{P}\{\mathbf{S}_1(\mathfrak{a})~\text{on boundary}~|~\|\mathbf{S}_{0}(\mathfrak{a})\|=n\},
\end{aligned}
\end{equation*}
\begin{equation*}\label{eq.1.32}
\begin{aligned}
\mu_n(\mathfrak{a})\asymp&\frac{1}{2}\mathsf{P}\{\mathbf{S}_1(\mathfrak{a})~\text{in interior}~|~\|\mathbf{S}_{0}(\mathfrak{a})\|=n\}\\
&+\frac{1}{4}\mathsf{P}\{\mathbf{S}_1(\mathfrak{a})~\text{on boundary}~|~\|\mathbf{S}_{0}(\mathfrak{a})\|=n\}.
\end{aligned}
\end{equation*}
Then, denoting $p_n(\mathfrak{a})=\mathsf{P}\{\mathbf{S}_1(\mathfrak{a})~\text{on boundary}~|~\|\mathbf{S}_{0}(\mathfrak{a})\|=n\}$ and keeping in mind that $\lambda_n(\mathfrak{a})+\mu_n(\mathfrak{a})=1$, we obtain
\begin{equation*}\label{eq.1.14}
\frac{P_0(\mathfrak{a}; n+1|n)}{P_0(\mathfrak{a}; n-1|n)}\asymp\frac{2+p_n(\mathfrak{a})}{2-p_n(\mathfrak{a})}\asymp1+p_n(\mathfrak{a}).
\end{equation*}
In Example \ref{ex1}, where $\mathfrak{a}=(\alpha, \alpha,\ldots)$, the result $\psi=\mathrm{e}^{8\alpha}$ means that for large $n$, $p_n(\mathfrak{a})$ is evaluated as $\mathrm{e}^{8\alpha}/n$.
The critical value of $\alpha$ separating recurrence and transience is zero. So, for positive $\alpha$ the random walk is transient, but for zero or negative $\alpha$ it is recurrent.

\smallskip

 Note, that the study of zero-drift random walks in a quarter plane by the method of constructing Lyapunov function made important to consider one-dimensional processes with asymptotically zero drifts. These problems were originally formulated by Harris \cite{Harris1, Harris2} and then comprehensively solved by Lamperti \cite{Lamperti1, Lamperti2}. Essential generalization of  \cite{Lamperti1, Lamperti2} has been provided in \cite{MAI} for the processes whose increments have moments of the order of $2+\epsilon$, $\epsilon>0$. The example below demonstrates the results on recurrence and transience for quite general state-dependent one-dimensional random walk.

\begin{example}
For similarly defined one-dimensional random walks we have as follows. Let $\mathfrak{a}=(\alpha_1, \alpha_2,\ldots)$, and let $S_0(\mathfrak{a})=1$,
$S_t(\mathfrak{a})=S_{t-1}(\mathfrak{a})+e_t(\mathfrak{a})$, $t\geq1$, where $e_t(\mathfrak{a})$ takes values $\pm1$, and the distribution of $S_t(\mathfrak{a})$ is defined by the following conditions:

\begin{eqnarray*}
  \mathsf{P}\{S_{t+1}(\mathfrak{a})=S_t(\mathfrak{a})+1~|~S_t(\mathfrak{a})>0\} &=& \frac{1}{2}+\frac{\alpha_{S_t(\mathfrak{a})}}{S_t(\mathfrak{a})},\label{eq.1.16} \\
  \mathsf{P}\{S_{t+1}(\mathfrak{a})=S_t(\mathfrak{a})-1~|~S_t(\mathfrak{a})>0\} &=& \frac{1}{2}-\frac{\alpha_{S_t(\mathfrak{a})}}{S_t(\mathfrak{a})}.\label{eq.1.17}\\
  \mathsf{P}\{S_{t+1}(\mathfrak{a})=1~|~S_t(\mathfrak{a})=0\}&=&1.
\end{eqnarray*}
The values $\alpha_n$ satisfy the condition
$$
0<\alpha_n<\min\left\{C, \frac{1}{2}n\right\}, \quad C>0.
$$
\end{example}

Then, the behaviour of this one-dimensional random walk is associated with the behaviour of the birth-and-death process with the birth rates $\lambda_n=1/2+\alpha_n/n$  and death rates $\mu_n=1/2-\alpha_n/n$ ($n\geq1$). It is well-known that a birth-and-death process is recurrent if and only if
\begin{equation}\label{eq.1.1}
\sum_{n=1}^\infty\prod_{k=1}^{n}\frac{\mu_k}{\lambda_k}=\infty
\end{equation}
see \cite{KM}, p. 370 or \cite{M_et_all}, Theorem 2.2.5. The conditions on the sequence $\alpha_n$ under which the random walk is recurrent or transient follow directly from the extended Bertrand--De Morgan test for convergence or divergence of positive series \cite{A3, M}, or extended Bertrand--De Morgan--Cauchy test \cite{A4}. For instance, from \cite{A3} we arrive at the following result.
Let  $\ln_{(K)} x$ denote the $K$th iterate of natural logarithm, i.e. $\ln_{(1)}x=\ln x$ and for $k\geq 2$, $\ln_{(k)}x=\ln_{(k-1)}(\ln x)$.
\begin{prop}\label{prop1}
The random walk is transient if there exist $c>1$, $K\geq1$ and $n_0$ such that for all $n>n_0$
$$
\alpha_n\geq \frac{1}{4}\left(1+\sum_{i=1}^{K-1}\frac{1}{\prod_{j=1}^{i}\ln_{(j)}n}+\frac{c}{\prod_{j=1}^{K}\ln_{(j)}n}\right),
$$
and is recurrent if there exist $K\geq1$ and $n_0$ such that for all $n>n_0$
$$
\alpha_n\leq \frac{1}{4}\left(1+\sum_{i=1}^{K}\frac{1}{\prod_{j=1}^{i}\ln_{(j)}n}\right).
$$
\end{prop}

For earlier particular results related to this random walk see \cite{Chung, Harris1, H, K}.

\subsection{Outline of the paper} The paper is organized as follows. In Section \ref{S2}, we provide theoretical grounds for modelling the family of random walks. That is, we construct a queueing network that models the family of random walks. In Section \ref{S3}, we write Chapman-Kolmogorov system of equations describing the dynamics of queue-length processes in the network. In Section \ref{S4}, we provide asymptotic study of queue-length processes, and on its basis prove the basic result of the paper. Lemma \ref{L3} of Section \ref{S4} plays key role in the proof.
In Section \ref{S5}, we conclude the paper with the discussion of the result and possible future research.

\section{Theoretical grounds for modelling the family of random walks}\label{S2}

\subsection{Prelude} The aim of this section is to represent the family of random walks in the form that is convenient for its further analytical study. To do that,
we assume that a random walk of the family stays an exponentially distributed random time in any of its states prior moving to another state, and all exponentially distributed times are independent and identically distributed. (This assumption is not the loss of generality. Some details explaining this construction can be found in \cite{A,A5, A6}.)
Then, the meaning of the time parameter $t$ is the $t$th event in a Poisson process. Advantage of the reduction to the continuous time process is that enables us to use technical simplicity of standard methods for continuous time Markov processes and queueing theory.

Therefore, analysis of relations \eqref{eq.1.4} -- \eqref{eq.1.13} reduces our study to an analysis of a state-dependent network of two $M/M/1$ queueing systems.

\subsection{Queueing models of the family of reflected random walks}
\subsubsection{Queueing model for reflected simple random walk}\label{S2.2.1}
Consider first the queueing model for the reflected simple two-dimensional random walk ${\mathbf{S}}_t(\mathfrak{o})$. The simple two-dimensional random walk belongs to the family of symmetric two-dimensional random walks \cite{A}. The queueing model of reflected symmetric random walks has been described in \cite{A}, and here we recall this construction for ${\mathbf{S}}_t(\mathfrak{o})$.  The aforementioned reflected random walk ${\mathbf{S}}_t(\mathfrak{o})$ is modelled with the aid of two independent and identical $M/M/1$ queueing systems with same arrival and service rates. If a system is free, then the server switches on \textit{negative} service having the exponential distribution with the same mean as interarrival or service time.
A negative service results a new customer in the system. If during a negative service a new customer arrives, then the negative service is interrupted and not resumed. In other words, the negative service is equivalent to the doubled arrival rate when a system is empty, which is the result of a reflection at zero. Analysis of the system of two queues led to the following results. Let $\mathbf{i}=(i^{(1)}, i^{(2)})$ denote a vector with integer nonnegative components, and let $P_\tau(\mathbf{i})$ denote the probability that at time $\tau$ there are $i^{(1)}$ customers in the first system and $i^{(2)}$ customers in the second one (the time parameter $\tau$ is assumed to be a continuous variable).

A result obtained in \cite{A} (see the proof of Theorem 2.1 in \cite{A}) for these queues is convenient to reformulate as follows
\begin{equation}\label{eq.2.1}
\lim_{\tau\to\infty}\frac{P_\tau(\mathbf{i})}{P_\tau(\mathbf{0})}=\begin{cases}4, &\text{if} \ i^{(1)}>0 \ \text{and} \ i^{(2)}>0,\\
2, &\text{if} \ i^{(1)}>0, i^{(2)}=0 \ \text{or} \ i^{(2)}>0, i^{(1)}=0.
\end{cases}
\end{equation}
Now, let $\mathcal{N}^+(n)$ denote the set of nonnegative vectors $\mathbf{i}$, the sum of the components of which is equal to $n$. Then, from \eqref{eq.2.1} we have
\begin{equation}\label{eq.2.2}
\lim_{\tau\to\infty}\frac{\sum_{\mathbf{i}\in\mathcal{N}^+(n)}P_\tau(\mathbf{i})}{P_\tau(\mathbf{0})}=4n.
\end{equation}
In turn, based on \eqref{eq.2.2} we then arrived at the conclusion, that
$$
\lambda_n(\mathfrak{o})=P_0(\mathfrak{o}; n+1|n)=\lim_{t\to\infty}\mathsf{P}\{\|\mathbf{S}_{t+1}(\mathfrak{o})\|=n+1~\big|~\|\mathbf{S}_{t}(\mathfrak{o})\|=n\},
$$
and
$$
\mu_n(\mathfrak{o})=1-\lambda_n(\mathfrak{o})
$$
are, respectively, the rates that are proportional to the birth and death rates of $BD(2,2)$, which is the birth-and-death process introduced in \cite{A}. The birth and death rates of $BD(2,2)$ are given by
\begin{eqnarray*}
  \lambda_n(2,2) &=& 8n+4, \\
  \mu_n(2,2)     &=& 8n-4,
\end{eqnarray*}
and
$$
\frac{\lambda_n(2,2)}{\mu_n(2,2)}=1+\frac{1}{n}+O\left(\frac{1}{n^2}\right),
$$
that constitutes null-recurrence of $BD(2,2)$ (see Lemma 4.2 in \cite{A}) and, hence, recurrence of $\mathbf{S}_t(\mathfrak{o})$.

\subsubsection{Queueing model for reflected random walks in general case}\label{S2.2.2}
In contrast to the case considered in Section \ref{S2.2.1}, the reflected random walk ${\mathbf{S}}_t(\mathfrak{a})$ is modelled as the network of two dependent and, in addition, state-dependent $M/M/1$ queueing systems as described below.
 As in the case considered above, each of these system, being free, switches to negative service having the exponential distribution with mean 1, and if during the negative service an arrival of a customer in the empty system occurs, the negative service is interrupted. The means of customer interarrival and service times depend on the network state as follows. Let $Q_\tau^{(1)}$ and $Q_\tau^{(2)}$ denote the number of customers in the first and, respectively, second queueing systems at time $\tau$ and let $N_\tau=Q_\tau^{(1)}+Q_\tau^{(2)}$. If both $Q_\tau^{(1)}$ and $Q_\tau^{(2)}$ are positive and $Q_\tau^{(1)}<Q_\tau^{(2)}$, then the mean interarrival times in the first and second systems are $1/\big(1-4\alpha_{(Q_\tau^{(1)},Q_\tau^{(2)})}/N_\tau\big)$ and $1/\big(1+4\alpha_{(Q_\tau^{(1)},Q_\tau^{(2)})}/N_\tau\big)$, respectively, and the mean service times are $1/\big(1+4\alpha_{(Q_\tau^{(1)},Q_\tau^{(2)})}/N_\tau\big)$ and $1/\big(1-4\alpha_{(Q_\tau^{(1)},Q_\tau^{(2)})}/N_\tau\big)$, respectively. If $Q_\tau^{(1)}>Q_\tau^{(2)}$, then the mean interarrival times in the first and second systems are $1/(1+4\alpha_{(Q_\tau^{(2)},Q_\tau^{(1)})}/N_\tau)$ and $1/(1-4\alpha_{(Q_\tau^{(2)},Q_\tau^{(1)})}/N_\tau)$, respectively, while the mean service times are $1/\big(1-4\alpha_{(Q_\tau^{(2)},Q_\tau^{(1)})}/N_\tau\big)$ and $1/\big(1+4\alpha_{(Q_\tau^{(2)},Q_\tau^{(1)})}/N_\tau\big)$, respectively. In the cases where $Q_\tau^{(1)}$ and $Q_\tau^{(2)}$ are positive and equal, the means of interarrival and service times of both systems is 1. If one system is empty, then the means of interarrival and service times in another system are equal to 1, and in the first system the mean interarrival time and mean negative service time both are equal to 1 as well.

When $\mathfrak{a}\neq \mathfrak{o}$, the network of queueing systems (which will be called later $\mathfrak{a}$-network) has a complex structure. The system of the equations for queue-length probabilities cannot be presented in product form, and the arguments that were earlier used for symmetric family of independent random walks and leading to exact representations become disable. So, the aim is to study structural properties of the distributions, and on the basis of them to arrive at the correct conclusions on the family of random walks.

Note that typical queueing networks, including those state-dependent, suppose the presence of routing mechanism between the queueing systems connected into the network. In \cite{MP} a large number of different examples of application of state-dependent queueing networks have been considered. The  state-dependent network of two queueing systems that models the random walks considered in the present paper does not contain an explicit connection between the queues via the routing mechanism.
The connection between the queues is implicitly provided via the dependence of arrival and service rates in both queueing systems upon the total number of customers presenting in two queues. Similar type of networks describe models of service allocations and have been studied in \cite{SS}.

\section{The system of equations for the state probabilities}\label{S3}

In this section, we derive equations for state probabilities of the $a$-network. We say that $a$-network is in state $\mathbf{n}=(n^{(1)}, n^{(2)})$, $n^{(1)}\leq n^{(2)}$, if there are $n^{(1)}$ customers in the smallest queue and $n^{(2)}$ customers in the largest one. This definition of state is not traditional. The use of this definition is motivated by the symmetry of a random walk (see Figure 1) and its reduction to the random walk in a wedge (see Figure 2) and enables us
to diminish the number of possible cases and, hence, the number of the
equations that describe the state probabilities, compared to that might be written in the case of traditional description of the system state. This, however, requires to use a specified algebra rule in derivation of the Chapman-Kolmogorov system of equations for the state probabilities that is explained in the case studies below (see for instance the comparison of the transition probabilities in Figures 1 and 2 for the boundary case). Note, that compared to the rates indicated in Section \ref{S2.2.2}, the rates in the equations below are given multiplied by factor 4.

Denote by $P_\tau\big(\mathfrak{a},\mathbf{n}\big)$ the probability that at time $\tau$ the $\mathfrak{a}$-network is in state $\mathbf{n}=(n^{(1)}, n^{(2)})$.
Then, the Chapman-Kolmogorov system of equations is as follows.

\begin{enumerate}
  \item \textit{Case} $n^{(1)}>1$ \textit{and} $n^{(2)}>n^{(1)}+1$:
  \begin{equation}\label{eq.3.1}
  \begin{aligned}
  &\frac{\mathrm{d}P_\tau(\mathfrak{a},\mathbf{n})}{\mathrm{d}\tau}+4P_\tau(\mathfrak{a},\mathbf{n})\\
  &=\left(1-\frac{4\alpha_{\mathbf{n}-\mathbf{1}_1}}{\|\mathbf{n}\|-1}\right)P_\tau(\mathfrak{a},\mathbf{n}-\mathbf{1}_1)\\
  +&\left(1+\frac{4\alpha_{\mathbf{n}-\mathbf{1}_2}}{\|\mathbf{n}\|-1}\right)P_\tau(\mathfrak{a},\mathbf{n}-\mathbf{1}_2)\\
  +&\left(1+\frac{4\alpha_{\mathbf{n}+\mathbf{1}_1}}{\|\mathbf{n}\|+1}\right)P_\tau(\mathfrak{a},\mathbf{n}+\mathbf{1}_1)\\
  +&\left(1-\frac{4\alpha_{\mathbf{n}+\mathbf{1}_2}}{\|\mathbf{n}\|+1}\right)P_\tau(\mathfrak{a},\mathbf{n}+\mathbf{1}_2).
  \end{aligned}
  \end{equation}
  Recall that $\mathbf{1}_1=(1,0)$ and $\mathbf{1}_2=(0,1)$.

  \smallskip
  \item \textit{Case} $n^{(1)}>1$ \textit{and} $n^{(2)}=n^{(1)}+1$:
  \begin{equation}\label{eq.3.3}
  \begin{aligned}
  &\frac{\mathrm{d}P_\tau(\mathfrak{a},\mathbf{n})}{\mathrm{d}\tau}+4P_\tau(\mathfrak{a},\mathbf{n})\\
  &=\left(1-\frac{4\alpha_{\mathbf{n}-\mathbf{1}_1}}{\|\mathbf{n}\|-1}\right)P_\tau(\mathfrak{a},\mathbf{n}-\mathbf{1}_1)\\
  +&2P_\tau(\mathfrak{a},\mathbf{n}-\mathbf{1}_2)
  +2P_\tau(\mathfrak{a},\mathbf{n}+\mathbf{1}_1)\\
  +&\left(1-\frac{4\alpha_{\mathbf{n}+\mathbf{1}_2}}{\|\mathbf{n}\|+1}\right)P_\tau(\mathfrak{a},\mathbf{n}+\mathbf{1}_2).
  \end{aligned}
  \end{equation}
  The terms $2P_\tau(\mathfrak{a},\mathbf{n}-\mathbf{1}_2)$ and $2P_\tau(\mathfrak{a},\mathbf{n}+\mathbf{1}_1)$ both enter with coefficient 2 since the vectors
  $\mathbf{n}-\mathbf{1}_2$ and $\mathbf{n}+\mathbf{1}_1$ have the first and second coordinates equal.

  \smallskip
  \item \textit{Case} $n^{(1)}=1$ \textit{and} $n^{(2)}>2$:
  \begin{equation}\label{eq.3.5}
  \begin{aligned}
  &\frac{\mathrm{d}P_\tau(\mathfrak{a},\mathbf{n})}{\mathrm{d}\tau}+4P_\tau(\mathfrak{a},\mathbf{n})\\
  &=2P_\tau(\mathfrak{a},\mathbf{n}-\mathbf{1}_1)\\
  +&\left(1+\frac{4\alpha_{\mathbf{n}-\mathbf{1}_2}}{\|\mathbf{n}\|-1}\right)P_\tau(\mathfrak{a},\mathbf{n}-\mathbf{1}_2)\\
  +&\left(1+\frac{4\alpha_{\mathbf{n}}+\mathbf{1}_1}{\|\mathbf{n}\|+1}\right)P_\tau(\mathfrak{a},\mathbf{n}+\mathbf{1}_1)\\
  +&\left(1-\frac{4\alpha_{\mathbf{n}+\mathbf{1}_2}}{\|\mathbf{n}\|+1}\right)P_\tau(\mathfrak{a},\mathbf{n}+\mathbf{1}_2).
  \end{aligned}
  \end{equation}
  The term $2P_\tau(\mathfrak{a},\mathbf{n}-\mathbf{1}_1)$ enters with coefficient 2 since the first coordinate of the vector $\mathbf{n}-\mathbf{1}_1$ is zero, and, according to convention, the rate of arrival and negative service together make the total rate doubled.

  \smallskip
  \item \textit{Case} $\mathbf{n}=(1,2)$:
  \begin{equation}\label{eq.3.7}
  \begin{aligned}
  &\frac{\mathrm{d}P_\tau(\mathfrak{a},\mathbf{n})}{\mathrm{d}\tau}+4P_\tau(\mathfrak{a},\mathbf{n})\\
  &=2P_\tau(\mathfrak{a},\mathbf{n}-\mathbf{1}_1)\\
  +&2P_\tau(\mathfrak{a},\mathbf{n}-\mathbf{1}_2)
  +2P_\tau(\mathfrak{a},\mathbf{n}+\mathbf{1}_1)\\
  +&\left(1-\alpha_{\mathbf{n}+\mathbf{1}_2}\right)P_\tau(\mathfrak{a},\mathbf{n}+\mathbf{1}_2).
  \end{aligned}
  \end{equation}
  The term $2P_\tau(\mathfrak{a},\mathbf{n}-\mathbf{1}_1)$ enters with coefficient 2 for the reason indicated for Case (3). The terms $2P_\tau(\mathfrak{a},\mathbf{n}-\mathbf{1}_2)$ and $2P_\tau(\mathfrak{a},\mathbf{n}+\mathbf{1}_1)$ enter with coefficient 2 for the reason indicated for Case (2).

  \smallskip
  \item \textit{Case} $n^{(1)}=n^{(2)}\geq2$:
    \begin{equation}\label{eq.3.9}
  \begin{aligned}
  &\frac{\mathrm{d}P_\tau(\mathfrak{a},\mathbf{n})}{\mathrm{d}\tau}+4P_\tau(\mathfrak{a},\mathbf{n})\\
  &=\left(1-\frac{4\alpha_{\mathbf{n}-\mathbf{1}_1}}{\|\mathbf{n}\|-1}\right)P_\tau(\mathfrak{a},\mathbf{n}-\mathbf{1}_1)\\
  +&\left(1-\frac{4\alpha_{\mathbf{n}+\mathbf{1}_2}}{\|\mathbf{n}\|+1}\right)P_\tau(\mathfrak{a},\mathbf{n}+\mathbf{1}_2).
  \end{aligned}
  \end{equation}

  \smallskip
  \item \textit{Case} $\mathbf{n}=(1,1)$:
  \begin{equation}\label{eq.3.10}
  \begin{aligned}
  &\frac{\mathrm{d}P_\tau(\mathfrak{a},\mathbf{n})}{\mathrm{d}\tau}+4P_\tau(\mathfrak{a},\mathbf{n})\\
  &=2P_\tau(\mathfrak{a},\mathbf{n}-\mathbf{1}_1)\\
  +&\left(1-\frac{4\alpha_{\mathbf{n}+\mathbf{1}_2}}{3}\right)P_\tau(\mathfrak{a},\mathbf{n}+\mathbf{1}_2).
  \end{aligned}
  \end{equation}
  The term $2P_\tau(\mathfrak{a},\mathbf{n}-\mathbf{1}_1)$ enters with coefficient 2 for the reason indicated for Case (3).

  \smallskip
  \item \textit{Case} $\mathbf{n}=(0,n)$ for $n\geq2$:
  \begin{equation}\label{eq.3.12}
  \begin{aligned}
  &\frac{\mathrm{d}P_\tau(\mathfrak{a},\mathbf{n})}{\mathrm{d}\tau}+4P_\tau(\mathfrak{a},\mathbf{n})\\
  &=P_\tau(\mathfrak{a},\mathbf{n}-\mathbf{1}_2)+P_\tau(\mathfrak{a},\mathbf{n}+\mathbf{1}_2)\\
  +&\left(1+\frac{4\alpha_{\mathbf{n}+\mathbf{1}_1}}{\|\mathbf{n}\|+1}\right)P_\tau(\mathfrak{a},\mathbf{n}+\mathbf{1}_1).
  \end{aligned}
  \end{equation}

  \smallskip
  \item \textit{Case} $\mathbf{n}=(0,1)$:
  \begin{equation}\label{eq.3.11}
  \begin{aligned}
  &\frac{\mathrm{d}P_\tau(\mathfrak{a},\mathbf{n})}{\mathrm{d}\tau}+4P_\tau(\mathfrak{a},\mathbf{n})\\
  &=8P_\tau(\mathfrak{a},\mathbf{0})+2P_\tau(\mathfrak{a},\mathbf{n}+\mathbf{1}_1)
  +P_\tau(\mathfrak{a},\mathbf{n}+\mathbf{1}_2)
  \end{aligned}
  \end{equation}
  The term $8P_\tau(\mathfrak{a},\mathbf{0})$ enters with coefficient 8 since in state $\mathbf{0}$ two arrival processes and two negative services together result in rate 4, and, in addition the two coordinates of vector $\mathbf{0}$ are equal. The term $2P_\tau(\mathfrak{a},\mathbf{n}+\mathbf{1}_1)$ enters with coefficient 2 for the reason indicated for Case (2).

  \smallskip
  \item \textit{Case} $\mathbf{n}=\mathbf{0}$:
   \begin{equation}\label{eq.3.13}
  \frac{\mathrm{d}P_\tau(\mathfrak{a},\mathbf{0})}{\mathrm{d}\tau}+4P_\tau(\mathfrak{a},\mathbf{0})=P_\tau(\mathfrak{a},\mathbf{1}_2).
  \end{equation}

\end{enumerate}

\smallskip
Our next step is the steady-state solution as $\tau$ increases to infinity. As $\tau\to\infty$, all the terms $\mathrm{d}P_\tau(\mathfrak{a},\mathbf{n})/\mathrm{d}\tau$ and $P_\tau(\mathfrak{a},\mathbf{n})$ vanish.
However the limit
\begin{equation}\label{eq.3.30}
p(\mathfrak{a},\mathbf{n})=\lim_{\tau\to\infty}\frac{P_\tau(\mathfrak{a},\mathbf{n})}{P_\tau(\mathfrak{a},\mathbf{0})}.
\end{equation}
exists and is positive. We could not find the relevant property in an available literature. So, we are to prove \eqref{eq.3.30}. The direct proof is based on exploiting
relations \eqref{eq.3.1} -- \eqref{eq.3.13} and is cumbersome. Therefore, we first provide it in
the simplest case for $p(n)$, where $n$ is scalar value and in the case when the
system is homogeneous. Then we explain how this proof is adapted to the
case considered in the paper.

Let $r$ be a fixed real positive parameter, and let
\begin{eqnarray}\label{eq.3.2}
\frac{\mathrm{d}P_\tau(r,0)}{\mathrm{d}\tau}&=&-\lambda(r)P_\tau(r,0)+\mu(r)P_\tau(r,1),\nonumber\\
\frac{\mathrm{d}P_\tau(r,n)}{\mathrm{d}\tau}&=&\lambda(r)P_\tau(r,n-1)
-(\lambda(r)+\mu(r))P_\tau(r,n)\\
&&+\mu(r)P_\tau(r,n+1)\nonumber
\end{eqnarray}
be a system of differential equations describing the evolution of an $M/M/1$ system. If $\lambda(r)<\mu(r)$, then there is $\lim_{\tau\to\infty}P_\tau(r,n)=p(r,n)$ and the explicit form of this limit is known. Assume that $\lambda(r)<\mu(r)$ for positive $r$, $\rho(r)=\lambda(r)/\mu(r)$ increases in $r$, and $\rho(r)\to1$ as $r\to\infty$. Then, for any $n\geq1$ the solution $p(r,n)$ satisfies the following property. If $r_1<r_2$, then
$$
\frac{p(r_1,n)}{p(r_1,0)}<\frac{p(r_2,n)}{p(r_2,0)}.
$$
That is, the fraction ${p(r,n)}/{p(r,0)}$ increases in $r$, and, hence, there is the limit of this fraction, as $r\to\infty$. This limit is finite, and in the specific case of this series of systems is known to be equal to 1. The convergence is uniform, and exchanging the order of limit $r$ vs $\tau$ is available. That is,
$$
p(n)=\lim_{\tau\to\infty}\lim_{r\to\infty}\frac{P_\tau(r,n)}{P_\tau(r,0)}
$$
exists and is positive.

To adapt this proof to the case of the present paper we also introduce
a positive parameter $r$, considering the system of equations for $P_\tau(\mathfrak{a}, r, \mathbf{n})$ rather than for $P_\tau(\mathfrak{a}, \mathbf{n})$. For large $\|\mathbf{n}\|$ introduce the parameters
$$
\lambda_{\mathrm{min}}(r, \mathbf{n})=\lambda(r)\left(1-\frac{4\alpha_{\mathbf{n}}}{\|\mathbf{n}\|}\right), \quad \lambda_{\mathrm{max}}(r, \mathbf{n})=\lambda(r)\left(1+\frac{4\alpha_{\mathbf{n}}}{\|\mathbf{n}\|}\right),
$$
and
$$
\mu_{\mathrm{min}}(r, \mathbf{n})\equiv\mu_{\mathrm{min}}(\mathbf{n})=1-\frac{4\alpha_{\mathbf{n}}}{\|\mathbf{n}\|}, \quad \mu_{\mathrm{max}}(r, \mathbf{n})\equiv\mu_{\mathrm{max}}(\mathbf{n})=1+\frac{4\alpha_{\mathbf{n}}}{\|\mathbf{n}\|},
$$
where $\lambda(r)$ is a positive increasing sequence, and $\lim_{r\to\infty}\lambda(r)=1$.

Then, we can obtain the system of equations similar to that \eqref{eq.3.2}. As
an example, below we present the only one of the equations of that system. Specifically, equation \eqref{eq.3.1} in our new settings looks
\begin{equation}\label{eq.3.4}
  \begin{aligned}
  &\frac{\mathrm{d}P_\tau(\mathfrak{a},r, \mathbf{n})}{\mathrm{d}\tau}+2(1+\lambda(r))P_\tau(\mathfrak{a}, r, \mathbf{n})\\
  &=\lambda_{\mathrm{min}}(r, \mathbf{n}-\mathbf{1}_1)P_\tau(\mathfrak{a}, r, \mathbf{n}-\mathbf{1}_1)\\
  +&\lambda_{\mathrm{max}}(r, \mathbf{n}-\mathbf{1}_2)P_\tau(\mathfrak{a}, r, \mathbf{n}-\mathbf{1}_2)\\
  +&\mu_{\mathrm{max}}(r, \mathbf{n}+\mathbf{1}_1)P_\tau(\mathfrak{a}, r, \mathbf{n}+\mathbf{1}_1)\\
  +&\mu_{\mathrm{min}}(r, \mathbf{n}+\mathbf{1}_2)P_\tau(\mathfrak{a}, r, \mathbf{n}+\mathbf{1}_2).
  \end{aligned}
  \end{equation}
The other equations of the aforementioned system of equations are similar.

Since $\lambda_{\mathrm{min}}(r, \mathbf{n})<\mu_{\mathrm{min}}(r, \mathbf{n})$ and $\lambda_{\mathrm{max}}(r, \mathbf{n})<\mu_{\mathrm{max}}(r, \mathbf{n})$, then there is
the limiting stationary distribution as $\tau\to\infty$, i.e. $\lim_{\tau\to\infty}P_{\tau}(\mathfrak{a}, r, \mathbf{n})=p(\mathfrak{a}, r, \mathbf{n})$. Then, similarly to the case considered above, for any $r_1<r_2$ we have
\[
\frac{p(\mathfrak{a}, r_1, \mathbf{n})}{p(\mathfrak{a}, r_1, \mathbf{0})}<\frac{p(\mathfrak{a}, r_2, \mathbf{n})}{p(\mathfrak{a}, r_2, \mathbf{0})}.
\]
Hence, there is the limit of the fraction $p(\mathfrak{a}, r, \mathbf{n})/p(\mathfrak{a}, r, \mathbf{0})$ as $r\to\infty$. Similarly to the above, this limit is uniform, and hence the order of limits $\tau$ vs $r$ can be exchanged. That is,
\[
\lim_{\tau\to\infty}\lim_{r\to\infty}P_{\tau}(\mathfrak{a}, r, \mathbf{n})=\lim_{r\to\infty}\lim_{\tau\to\infty}P_{\tau}(\mathfrak{a}, r, \mathbf{n}),
\]
and
\[
\lim_{\tau\to\infty}\frac{P_{\tau}(\mathfrak{a},\mathbf{n})}{P_{\tau}(\mathfrak{a},\mathbf{0})}
\]
exists and is positive.
\smallskip

Correspondingly to the above Cases (1) -- (9) equations, we obtain the following system of equations.

\begin{enumerate}
  \item \textit{Case} $n^{(1)}>1$ \textit{and} $n^{(2)}>n^{(1)}+1$:
  \begin{equation}\label{eq.3.14}
  \begin{aligned}
  p(\mathfrak{a},\mathbf{n})&=\frac{1}{4}\left[\left(1-\frac{4\alpha_{\mathbf{n}-\mathbf{1}_1}}{\|\mathbf{n}\|-1}\right)p(\mathfrak{a},\mathbf{n}-\mathbf{1}_1)\right.\\
  +&\left(1+\frac{4\alpha_{\mathbf{n}-\mathbf{1}_2}}{\|\mathbf{n}\|-1}\right)p(\mathfrak{a},\mathbf{n}-\mathbf{1}_2)\\
  +&\left(1+\frac{4\alpha_{\mathbf{n}+\mathbf{1}_1}}{\|\mathbf{n}\|+1}\right)p(\mathfrak{a},\mathbf{n}+\mathbf{1}_1)\\
  +&\left.\left(1-\frac{4\alpha_{\mathbf{n}+\mathbf{1}_2}}{\|\mathbf{n}\|+1}\right)p(\mathfrak{a},\mathbf{n}+\mathbf{1}_2)\right].
  \end{aligned}
  \end{equation}

  \smallskip
  \item \textit{Case} $n^{(1)}>1$ \textit{and} $n^{(2)}=n^{(1)}+1$:
  \begin{equation}\label{eq.3.16}
  \begin{aligned}
  p(\mathfrak{a},\mathbf{n})&=\frac{1}{4}\left[\left(1-\frac{4\alpha_{\mathbf{n}-\mathbf{1}_1}}{\|\mathbf{n}\|-1}\right)p(\mathfrak{a},\mathbf{n}-\mathbf{1}_1)\right.\\
  +&2p(\mathfrak{a},\mathbf{n}-\mathbf{1}_2)
  +2p(\mathfrak{a},\mathbf{n}+\mathbf{1}_1)\\
  +&\left.\left(1-\frac{4\alpha_{\mathbf{n}+\mathbf{1}_2}}{\|\mathbf{n}\|+1}\right)p(\mathfrak{a},\mathbf{n}+\mathbf{1}_2)\right].
  \end{aligned}
  \end{equation}

  \smallskip
  \item \textit{Case} $n^{(1)}=1$ \textit{and} $n^{(2)}>2$:
  \begin{equation}\label{eq.3.18}
  \begin{aligned}
  p(\mathfrak{a},\mathbf{n})&=\frac{1}{4}\bigg[2P_\tau(\mathfrak{a},\mathbf{n}-\mathbf{1}_1)\\
  +&\left(1+\frac{4\alpha_{\mathbf{n}-\mathbf{1}_2}}{\|\mathbf{n}\|-1}\right)p(\mathfrak{a},\mathbf{n}-\mathbf{1}_2)\\
  +&\left(1+\frac{4\alpha_{\mathbf{n}+\mathbf{1}_1}}{\|\mathbf{n}\|+1}\right)p(\mathfrak{a},\mathbf{n}+\mathbf{1}_1)\\
  +&\left.\left(1-\frac{4\alpha_{\mathbf{n}+\mathbf{1}_2}}{\|\mathbf{n}\|+1}\right)p(\mathfrak{a},\mathbf{n}+\mathbf{1}_2)\right].
  \end{aligned}
  \end{equation}

  \smallskip
  \item \textit{Case} $\mathbf{n}=(1,2)$:
  \begin{equation*}\label{eq.3.20}
  \begin{aligned}
  p(\mathfrak{a},\mathbf{n})&=\frac{1}{4}\bigg[2p(\mathfrak{a},\mathbf{n}-\mathbf{1}_1)\\
  +&2p(\mathfrak{a},\mathbf{n}-\mathbf{1}_2)
  +2p(\mathfrak{a},\mathbf{n}+\mathbf{1}_1)\\
  +&\left.\left(1-\alpha_{\mathbf{n}+\mathbf{1}_2}\right)p(\mathfrak{a},\mathbf{n}+\mathbf{1}_2)\right].
  \end{aligned}
  \end{equation*}

  \smallskip
  \item \textit{Case} $n^{(1)}=n^{(2)}\geq2$:
    \begin{equation}\label{eq.3.22}
  \begin{aligned}
  p(\mathfrak{a},\mathbf{n})&=\frac{1}{4}\bigg[\left(1-\frac{4\alpha_{\mathbf{n}-\mathbf{1}_1}}{\|\mathbf{n}\|-1}\right)p(\mathfrak{a},\mathbf{n}-\mathbf{1}_1)\\
  +&\left.\left(1-\frac{4\alpha_{\mathbf{n}+\mathbf{1}_2}}{\|\mathbf{n}\|+1}\right)p(\mathfrak{a},\mathbf{n}+\mathbf{1}_2)\right].
  \end{aligned}
  \end{equation}

  \smallskip
  \item \textit{Case} $\mathbf{n}=(1,1)$:
  \begin{equation*}\label{eq.3.23}
  \begin{aligned}
  p(\mathfrak{a},\mathbf{n})&=\frac{1}{4}\bigg[2p(\mathfrak{a},\mathbf{n}-\mathbf{1}_1)\\
  +&\left.\left(1-\frac{4\alpha_{\mathbf{n}+\mathbf{1}_2}}{3}\right)p(\mathfrak{a},\mathbf{n}+\mathbf{1}_2)\right].
  \end{aligned}
  \end{equation*}

  \smallskip
  \item \textit{Case} $\mathbf{n}=(0,n)$ for $n\geq2$:
  \begin{equation}\label{eq.3.25}
  \begin{aligned}
  p(\mathfrak{a},\mathbf{n})&=\frac{1}{4}\bigg[p(\mathfrak{a},\mathbf{n}-\mathbf{1}_2)+p(\mathfrak{a},\mathbf{n}+\mathbf{1}_2)\\
  +&\left.\left(1+\frac{4\alpha_{\mathbf{n}+\mathbf{1}_1}}{\|\mathbf{n}\|+1}\right)p(\mathfrak{a},\mathbf{n}+\mathbf{1}_1)\right].
  \end{aligned}
  \end{equation}

  \smallskip
  \item \textit{Case} $\mathbf{n}=(0,1)$:
  \begin{equation}\label{eq.3.24}
  p(\mathfrak{a},\mathbf{n})=\frac{1}{4}\left[8p(\mathfrak{a},\mathbf{0})+2p(\mathfrak{a},\mathbf{n}+\mathbf{1}_1)
  +p(\mathfrak{a},\mathbf{n}+\mathbf{1}_2)\right].
  \end{equation}

  \smallskip
  \item \textit{Case} $\mathbf{n}=\mathbf{0}$:
   \begin{equation}\label{eq.3.26}
  p(\mathfrak{a},\mathbf{0})=\frac{1}{4}p(\mathfrak{a},\mathbf{1}_2)=1.
  \end{equation}
\end{enumerate}

In the particular case $\mathfrak{a}=\mathfrak{o}$, the representation for $p[\mathfrak{o}, (i,j)]$ has been derived explicitly in \cite{A}. Specifically, ${\mathbf{S}}_\tau(\mathfrak{o})$ denote a continuous time version of the reflected simple random walk that is understood as follows.  According to the convention, a discrete time $t$ in the notation ${\mathbf{S}}_t(\mathfrak{o})$ denote the $t$th event of a Poisson process. Then, in the new time scale the reflected simple random walk ${\mathbf{S}}_t(\mathfrak{o})$ can be extended to the continuous time process ${\mathbf{S}}_\tau(\mathfrak{o})$, where $\tau$ is a continuous time.
Let $\mathbf{n}_1=\big(n^{(1)}_1,n^{(2)}_1\big)$ and
$\mathbf{n}_2=\big(n^{(1)}_2,n^{(2)}_2\big)$ be two states (recall that $n^{(1)}_1\leq n^{(2)}_1$ and $n^{(1)}_2\leq n^{(2)}_2$), and let $n^{(1)}_1>0$ and $n^{(2)}_1>n^{(1)}_1$.
Denote
\begin{equation}\label{eq.3.29}
R(\mathfrak{o}, \mathbf{n}_1, \mathbf{n}_2)=\lim_{\tau\to\infty}\frac{\mathsf{P}\{{\mathbf{S}}_\tau(\mathfrak{o})=\mathbf{n}_1\}}{\mathsf{P}\{{\mathbf{S}}_\tau(\mathfrak{o})=\mathbf{n}_2\}}.
\end{equation}
Following the result in  \cite{A}, the ratio in \eqref{eq.3.29} can be calculated explicitly. Indeed, the assumption $n^{(1)}_1>0$ implies $n^{(2)}_1>0$ since according to convention $n^{(1)}_1\leq n^{(2)}_1$. So, the both coordinates of a vector $\mathbf{n}_1$ are positive. According to the other assumption $n^{(2)}_1>n^{(1)}_1$, the coordinates of a vector $\mathbf{n}_1$ are distinct.

If the same is true for a vector $\mathbf{n}_2$ i.e. its coordinates are positive and distinct, then $R(\mathfrak{o}, \mathbf{n}_1, \mathbf{n}_2)=1$. If both coordinates of a vector $\mathbf{n}_2$ are positive but equal, then $R(\mathfrak{o}, \mathbf{n}_1, \mathbf{n}_2)=2$. Another possible case is where the first coordinate of a vector $\mathbf{n}_2$ is zero but the second one is positive. In that case $R(\mathfrak{o}, \mathbf{n}_1, \mathbf{n}_2)=2$ too. In the last case where $\mathbf{n}_2=\mathbf{0}$, we have $R(\mathfrak{o}, \mathbf{n}_1, \mathbf{n}_2)=8$. So,
\begin{equation}\label{eq.3.28}
R(\mathfrak{o}, \mathbf{n}_1, \mathbf{n}_2)=\begin{cases}1, &\text{if} \ n^{(1)}_2>0 \ \text{and} \ n^{(1)}_2\neq n^{(2)}_2,\\
2, &\text{if} \ n^{(1)}_2>0 \ \text{and} \ n^{(1)}_2=n^{(2)}_2,\\
2, &\text{if} \ n^{(1)}_2=0, n^{(2)}_2>0,\\
8, &\text{if} \ n^{(1)}_2=0 \ \text{and} \ n^{(2)}_2=0.
\end{cases}
\end{equation}

In terms of the notation $p[\mathfrak{o}, (i,j)]$ in \eqref{eq.3.30}, relation \eqref{eq.3.28} is equivalent to
\begin{equation}\label{eq.3.27}
p[\mathfrak{o}, (i,j)]=\begin{cases}1, &\text{for} \ i=0 \ \text{and} \ j=0,\\
4, &\text{for} \ i>0 \ \text{and} \ i=j,\\
4, &\text{for} \ i=0 \ \text{and} \ j>0,\\
8, &\text{for} \ \text{all other cases}.
\end{cases}
\end{equation}

\section{Asymptotic analysis as $\|\mathbf{n}\|$ increases to infinity and the proof of Theorem \ref{T1}}\label{S4}
\subsection{Preliminaries}\label{S4.1}

 As $n\to\infty$, from the explicit representations for $p(\mathfrak{a},\mathbf{n})$ we will derive asymptotic expansions. For this purpose, we use the notation $p(\mathfrak{a},\mathbf{n})=p[\mathfrak{a},(n^{(1)},n^{(2)})],$ $n^{(1)}+n^{(2)}=\|\mathbf{n}\|=n$.
If $n$ is large, and $i$ and $j$ are integer numbers, $i+j=n$, $1<i< j-1$, then the basic asymptotic properties that are essentially used below are as follows:
\begin{equation}\label{eq.4.1}
\begin{aligned}
  &p[\mathfrak{a},(i,j)]=\frac{1}{4}\left(1-\frac{4\alpha_{(i-1, j)}}{n}\right)p[\mathfrak{a},(i-1, j)]\\
  &\ \ \ +\frac{1}{4}\left(1-\frac{4\alpha_{(i, j+1)}}{n}\right)p[\mathfrak{a},(i, j+1)]\\
  &\ \ \ +\frac{1}{4}\left(1+\frac{4\alpha_{(i+1, j)}}{n}\right)p[\mathfrak{a},(i+1, j)]\\
  &\ \ \ +\frac{1}{4}\left(1+\frac{4\alpha_{(i, j-1)}}{n}\right)p[\mathfrak{a},(i, j-1)]\\
  & \ \ \ \ +O\left(\frac{1}{n^2}\right),
\end{aligned}
\end{equation}
\begin{equation}\label{eq.4.2}
\begin{aligned}
  &p[\mathfrak{a},(i,i+1)]=\frac{1}{4}\left(1-\frac{4\alpha_{(i-1,i+1)}}{n}\right)p[\mathfrak{a},(i-1,i+1)]\\
  &\ \ \ +\frac{1}{4}\left(1-\frac{4\alpha_{(i, i+2)}}{n}\right)p[\mathfrak{a},(i, i+2)]\\
  &\ \ \ +\frac{1}{2}p[\mathfrak{a}, (i,i)]+\frac{1}{2}p[\mathfrak{a},(i+1,i+1)]\\
  & \ \ \ \ +O\left(\frac{1}{n^2}\right),
\end{aligned}
\end{equation}
where $n$ in \eqref{eq.4.2} is assumed to be equal to $2i+1$, and
\begin{equation}\label{eq.4.3}
\begin{aligned}
  &p[\mathfrak{a},(i,i)]=\frac{1}{2}\left(1-\frac{4\alpha_{(i-1,i)}}{n}\right)p[\mathfrak{a},(i-1,i)]\\
  &\ \ \ +\frac{1}{2}\left(1-\frac{4\alpha_{(i, i+1)}}{n}\right)p[\mathfrak{a},(i, i+1)]\\
  & \ \ \ \ +O\left(\frac{1}{n^2}\right),
\end{aligned}
\end{equation}
where $n$ in \eqref{eq.4.3} is assumed to be equal to $2i$.

Asymptotic relations \eqref{eq.4.1}, \eqref{eq.4.2} and \eqref{eq.4.3} follow directly from \eqref{eq.3.14}, \eqref{eq.3.16} and \eqref{eq.3.22}, respectively. Specifically, the presence of the remainder $O(1/n^2)$ is explained by replacing the terms $n-1$ or $n+1$ in the denominators by $n$ when $n$ is large. It will be shown below that under conditions \eqref{eq.1.20} and \eqref{eq.1.21}, these asymptotic relations can be respectively presented as follows:
\begin{equation}\label{eq.4.40}
\begin{aligned}
  &p[\mathfrak{a},(i,j)]=\frac{1}{2}\left(1-\frac{4\alpha^*_{j-i+1}}{n}\right)p[\mathfrak{a},(i-1, j)]\\
  &\ \ \ +\frac{1}{2}\left(1+\frac{4\alpha^*_{j-i-1}}{n}\right)p[\mathfrak{a},(i, j-1)]\\
  & \ \ \ \ +O\left(\frac{1}{n^2}\right),
\end{aligned}
\end{equation}
\begin{equation}\label{eq.4.41}
\begin{aligned}
  &p[\mathfrak{a},(i,i+1)]=\frac{1}{2}\left(1-\frac{4\alpha^*_{2}}{n}\right)p[\mathfrak{a},(i-1,i+1)]+p[\mathfrak{a},(i,i)]\\
  & \ \ \ \ +O\left(\frac{1}{n^2}\right),
\end{aligned}
\end{equation}
and
\begin{equation}\label{eq.4.42}
\begin{aligned}
  &p[\mathfrak{a},(i,i)]=\left(1-\frac{4\alpha^*_{1}}{n}\right)p[\mathfrak{a},(i-1,i)]+O\left(\frac{1}{n^2}\right),
\end{aligned}
\end{equation}
where in all these asymptotic relations $\alpha^*_{j-i}=\lim_{k\to\infty}\alpha_{(k+i,k+j)}$.

For justification of these asymptotic expansions and other properties, it is convenient to change the measure $p$ as follows:
\begin{equation}\label{eq.4.20}
q[\mathfrak{a}, (i,j)]=\begin{cases}p[\mathfrak{a}, (i,j)], &\text{if} \ i\geq1 \ \text{and} \ i\neq j,\\
2p[\mathfrak{a}, (0,n)], &\text{if} \ i=0  \ \text{and} \ n>0,\\
2p\left[\mathfrak{a}, \left(\frac{n}{2},\frac{n}{2}\right)\right], &\text{if} \ i=j \ (n \ \text{is even}),\\
8p[\mathfrak{a}, (0,0)], &\text{if} \ i=j=0.
\end{cases}
\end{equation}

We demonstrate first that the values $q[\mathfrak{a}, \mathbf{n}]$ can be recurrently obtained from the system of recurrence relations. For this purpose, we need to order the vectors $\mathbf{n}=(n^{(1)},n^{(2)})$ in the way establishing these recurrence relations for $q[\mathfrak{a}, \mathbf{n}]$. First, we rank vectors $\mathbf{n}$ as follows:
\begin{equation}\label{eq.4.19}
\underbrace{(0,0)}_{\text{rank} \ 0}, \underbrace{(0,1)}_{\text{rank} \ 1}, \underbrace{(0,2), (1,1)}_{\text{rank} \ 2}, \underbrace{(0,3), (1,2)}_{\text{rank} \ 3},\ldots.
\end{equation}
So, the rank is the sum of the first and second components of the vector.

Then, the corresponding quantities $\alpha_{\mathbf{n}}$ and $q[\mathfrak{a}, \mathbf{n}]$ are ranked accordingly. For instance,
$$
\underbrace{q[\mathfrak{a},(0,0)]}_{\text{rank} \ 0}, \underbrace{q[\mathfrak{a},(0,1)]}_{\text{rank} \ 1}, \underbrace{q[\mathfrak{a},(0,2)], q[\mathfrak{a}, (1,1)]}_{\text{rank} \ 2}, \underbrace{q[\mathfrak{a},(0,3)], q[\mathfrak{a},(1,2)]}_{\text{rank} \ 3},\ldots.
$$
Note, that the total number of elements of rank $n$ is $1+\lfloor n/2\rfloor$, where $\lfloor n/2\rfloor$ denotes the integer part of $n/2$.
The ranking is helpful to provide the proof of some results by mathematical induction. For instance, we prove that $\sum_{i+j=n,i\leq j}q[\mathfrak{o}, (i,j)]=8(1+\lfloor n/2\rfloor)$.

Indeed, from \eqref{eq.3.26} and \eqref{eq.4.20}  $q[\mathfrak{o}, (0,0)]=q[\mathfrak{o}, (0,1)]=8$. Then, from \eqref{eq.3.24} and \eqref{eq.4.20}
$$
q[\mathfrak{o}, (0,1)]=\frac{1}{4}\bigg[2q[\mathfrak{o}, (0,0)]+q[\mathfrak{o}, (0,2)]+q[\mathfrak{o}, (1,1)]\bigg].
$$
That is,
$$
q[\mathfrak{o}, (0,2)]+q[\mathfrak{o}, (1,1)]=16.
$$
Using mathematical induction, we then assume that for ranks $n-2$ and $n-1$ we have $\sum_{i+j=n-2,i\leq j}q[\mathfrak{o}, (i,j)]=8(1+\lfloor (n-2)/2\rfloor)$ and, respectively,
$\sum_{i+j=n-1,i\leq j}q[\mathfrak{o}, (i,j)]=8(1+\lfloor (n-1)/2\rfloor)$, and using the relations that connect the elements of rank $n$ with the elements of rank $n-1$ and $n-2$ we arrive at the required result by counting the number of elements of rank $n-1$ and $n-2$ in the relations and multiplying by factor 8.

Using the same approach, we are able to derive relationship between
$$
\sum_{\substack{ {i+j=n},\\ {i\leq j}}}q[\mathfrak{a}, (i,j)]
$$
and the sums
$$
\sum_{\substack{i+j=n-1,\\i\leq j}}q[\mathfrak{a}, (i,j)] \quad \text{and} \quad \sum_{\substack{i+j=n-2,\\i\leq j}}q[\mathfrak{a}, (i,j)].
$$

To obtain the stronger results in the form of recurrence relations for $q[\mathfrak{a}, (i,j)]$ that are needed for the purpose of the present paper, we are to establish the order between the pairs $(i,j)$. Specifically, the ranked elements can be ordered as follows. If $n=i+j$ is even, then the suggested order of elements is
$$
q[\mathfrak{a}, (0,i+j)], q[\mathfrak{a}, (1,i+j-1)],\ldots,q\left[\mathfrak{a}, \left(\frac{i+j}{2},\frac{i+j}{2}\right)\right].
$$
That is, the last element has the two equal index components $(i+j)/2$.
If $n$ is odd, then the suggested order is
$$
q[\mathfrak{a}, (0,i+j)], q[\mathfrak{a}, (1,i+j-1)],\ldots,q\left[\mathfrak{a}, \left(\frac{i+j-1}{2},\frac{i+j+1}{2}\right)\right],
$$
where the last element has the two index components $(i+j-1)/2$ and $(i+j+1)/2$ that are distinguished by a unit.

Then, all the elements $q[\mathfrak{a}, (i,j)]$ with $i+j=n$ can be presented in the form of recurrence relation except the marginal element $q[\mathfrak{a}, (0,i+j)]$. For instance, for the presentation of the element $q[\mathfrak{a}, (i,i)]$
we use the relation
\begin{equation*}
\begin{aligned}
q[\mathfrak{a}, (i-1,i)]&=\frac{1}{4}\left(1-\frac{2\alpha_{(i-2,i)}}{i-1}\right)q[\mathfrak{a}, (i-2,i)]\\
+&\frac{1}{4}\left(1+\frac{2\alpha_{(i-1,i+1)}}{i}\right)q[\mathfrak{a}, (i-1,i+1)]\\
+&\frac{1}{4}q[\mathfrak{a}, (i-1,i-1)]+\frac{1}{4}q[\mathfrak{a}, (i,i)]
\end{aligned}
\end{equation*}
(see \eqref{eq.3.16}), where after rearranging the terms $q[\mathfrak{a}, (i,i)]$ is expressed via $q[\mathfrak{a}, (i-1,i+1)]$ of rank $2i$, $q[\mathfrak{a}, (i-1,i)]$ of rank $2i-1$ and $q[\mathfrak{a}, (i-1,i-1)]$ and $q[\mathfrak{a}, (i-2,i)]$ of rank $2i-2$. However, the marginal element $q[\mathfrak{a}, (0, 2i)]$ cannot be presented via the elements of the smaller ranks $2i-1$ and $2i-2$ only. For instance, from \eqref{eq.3.25} and \eqref{eq.4.20} we obtain
\begin{equation*}
\begin{aligned}
q[\mathfrak{a}, (0, 2i-1)]&=\frac{1}{2}\left(1+\frac{2\alpha_{(1,2i-1)}}{i}\right)q[\mathfrak{a}, (1, 2i-1)]\\
+&\frac{1}{4}q[\mathfrak{a}, (0, 2i-2)]
+\frac{1}{4}q[\mathfrak{a}, (0, 2i)].
\end{aligned}
\end{equation*}
That is, after rearranging the elements we can see that the element $q[\mathfrak{a}, (0, 2i)]$ is presented via the subsequent element $q[\mathfrak{a}, (1, 2i-1)]$ of the same rank  and the other two elements $q[\mathfrak{a}, (0, 2i-2)]$ and $q[\mathfrak{a}, (0, 2i-1)]$ of the lower ranks.

Thus, for all elements of rank $i+j$ except the element $q[\mathfrak{a}, (0,i+j)]$, we can derive the system of recurrence relations expressing all elements $q[\mathfrak{a}, (i,j)]$ via the preceding elements. Then, to resolve the problem for $q[\mathfrak{a}, (0,i+j)]$, the additional equation for $\sum_{i+j,i\leq j}q[\mathfrak{a}, (i,j)]$ is used, and hence, all the elements $q[\mathfrak{a}, (i,j)]$ can be uniquely determined from the recursion in the combination with the normalization condition. For instance, using mathematical induction, we now can prove that $q[\mathfrak{o}, (i,j)]\equiv 8$ for all $i\leq j$, confirming thus \eqref{eq.3.27} (or \eqref{eq.3.28}) obtained by the other way in \cite{A}.

Assumption \eqref{eq.1.21} enables us to prove that, as $k\to\infty$, the quantities $q[\mathfrak{a}, (k-l,k+l)]$ (if $n=k-l+k+l=2k$ is even) and $q[\mathfrak{a}, (k-l+1,k+l)$ (if $n=k-l+1+k+l=2k+1$ is odd) converge to the limits. Indeed, as $k\to\infty$, for $q[\mathfrak{a}, (k-l,k+l)]$, $l\geq1$, we have
\begin{equation}\label{eq.3.40}
\begin{aligned}
q[\mathfrak{a},(k-l,k+l)]&=\frac{1}{4}\left(1-\frac{2\alpha_{(k-l-1,k+l)}}{k}\right)\\
&\ \ \ \times q[\mathfrak{a},(k-l-1,k+l)]\\
+&\frac{1}{4}\left(1-\frac{2\alpha_{(k-l,k+l+1)}}{k}\right)\\
&\ \ \ \times q[\mathfrak{a},(k-l,k+l+1)]\\
+&\frac{1}{4}\left(1+\frac{2\alpha_{(k-l+1,k+l)}}{k}\right)\\
&\ \ \ \times q[a,(k-l+1,k+l)]\\
+&\frac{1}{4}\left(1-\frac{2\alpha_{(k-l,k+l-1)}}{k}\right)\\
&\ \ \ \times q[\mathfrak{a},(k-l,k+l-1)]\\
&\ \ \ \ +O\left(\frac{1}{k^2}\right).
\end{aligned}
\end{equation}
Similarly, for $q[\mathfrak{a}, (k-l+1,k+l+1)]$ we have
\begin{equation}\label{eq.3.41}
\begin{aligned}
q[\mathfrak{a},(k-l+1,k+l+1)]&=\frac{1}{4}\left(1-\frac{2\alpha_{(k-l,k+l+1)}}{k}\right)\\
&\ \ \ \times q[\mathfrak{a},(k-l,k+l+1)]\\
+&\frac{1}{4}\left(1-\frac{2\alpha_{(k-l+1,k+l+2)}}{k}\right)\\
&\ \ \ \times q[\mathfrak{a},(k-l+1,k+l+2)]\\
+&\frac{1}{4}\left(1+\frac{2\alpha_{(k-l+2,k+l+1)}}{k}\right)\\
&\ \ \ \times q[\mathfrak{a},(k-l+2,k+l+1)]\\
+&\frac{1}{4}\left(1-\frac{2\alpha_{(k-l+1,k+l)}}{k}\right)\\
&\ \ \ \times q[\mathfrak{a},(k-l+1,k+l)]\\
&\ \ \ \ +O\left(\frac{1}{k^2}\right).
\end{aligned}
\end{equation}
According to assumption \eqref{eq.1.21}, the differences between the coefficients $\alpha_{(k-l-1,k+l)}$, $\alpha_{(k-l,k+l+1)}$, $\alpha_{(k-l+1,k+l)}$, $\alpha_{(k-l,k+l-1)}$ in \eqref{eq.3.40} and the corresponding coefficients $\alpha_{(k-l,k+l+1)}$, $\alpha_{(k-l+1,k+l+2)}$, $\alpha_{(k-l+2,k+l+1)}$, $\alpha_{(k-l+1,k+l)}$ in \eqref{eq.3.41} vanish, and
$$
|\alpha_{(k-l,k+l+1)}-\alpha_{(k-l-1,k+l)}|\leq\gamma|\alpha_{(k-l-1,k+l)}-\alpha_{(k-l-2,k+l-1)}|,
$$
and
$$
|\alpha_{(k-l+2,k+l+1)}-\alpha_{(k-l+1,k+l)}|\leq\gamma|\alpha_{(k-l+1,k+l)}-\alpha_{(k-l,k+l-1)}|.
$$
As well,
there is the recurrence relation connecting $q[\mathfrak{a},(k-l+1,k+l+1)]$ and $q[\mathfrak{a},(k-l,k+l)]$. Although the recurrence relation is indirect, it satisfies the required properties that enable us to use the fixed point solution. The fixed point solution enables us to justify the inequality
\begin{equation}\label{eq.3.42}
\begin{aligned}
&|q[\mathfrak{a}, (k-l+1,k+l+1)]-q[\mathfrak{a}, (k-l,k+l)]|\\
&\leq\gamma|q[\mathfrak{a}, (k-l,k+l)]-q[\mathfrak{a}, (k-l-1,k+l-1)]|
\end{aligned}
\end{equation}
and, hence, the convergence of $q[\mathfrak{a}, (k-l,k+l)]$ to the limit as $k\to\infty$. Moreover, owing to monotonicity \eqref{eq.3.42}, together with the convergence of $q[\mathfrak{a}, (k-l,k+l)]$ as $k\to\infty$ we also arrive at the estimate
\begin{equation*}
q[\mathfrak{a}, (k-l+1,k+l+1)]-q[\mathfrak{a}, (k-l,k+l)]=O\left(\gamma^{2k}\right),
\end{equation*}
which justifies asymptotic expansion \eqref{eq.4.40}. Asymptotic expansions \eqref{eq.4.41} and \eqref{eq.4.42} are established similarly.

\subsection{Lemmas} Let $(i, j)$ be a pair of integers, $i\leq j$ and $i+j=n$.
If $n$ is even ($n=2k$, where $k$ is a positive integer), then $i$ and $j$ can be presented as $i=k-l+1$ and $j=k+l-1$. If $n$ is odd ($n=2k+1$), then $i$ and $j$ can be presented as $i=k-l+1$ and $j=k+l$. (In both cases $l$ is assumed to be a positive integer.) In the lemma below we derive the asymptotic representations for $q[\mathfrak{a}, (k-l, k+l)]$ via $q[\mathfrak{a}, (k-l+1, k+l)]$ ($1\leq l\leq k-1$) and for $q[\mathfrak{a}, (k-l+1, k+l)]$ via $q[\mathfrak{a}, (k-l+1, k+l-1)]$ ($2\leq l\leq k$).

\begin{lem}\label{L1} As $k\to\infty$, we have the following asymptotic expansions. For $1\leq l\leq k-1$,
\begin{equation}\label{eq.4.70}
\begin{aligned}
q[\mathfrak{a}, (k-l,k+l)]&=\left(1+\frac{2\alpha^*_{2l-1}}{k}+\frac{2\alpha^*_{2l}}{k}\right)\\
\times&q[\mathfrak{a}, (k-l+1,k+l)]+O\left(\frac{1}{k^2}\right),
\end{aligned}
\end{equation}
and for $2\leq l\leq k$,
\begin{equation}\label{eq.4.71}
\begin{aligned}
q[\mathfrak{a}, (k-l+1,k+l)]&=\left(1+\frac{2\alpha^*_{2l-2}}{k}+\frac{2\alpha^*_{2l-1}}{k}\right)\\
\times&q[\mathfrak{a}, (k-l+1,k+l-1)]+O\left(\frac{1}{k^2}\right).
\end{aligned}
\end{equation}
\end{lem}
\begin{proof}
The proofs of \eqref{eq.4.70}, \eqref{eq.4.71} are based on mathematical induction. Hence, the first task is to prove the result for the initial value $l=1$ for \eqref{eq.4.70} and for the initial value $l=2$ for \eqref{eq.4.71}.

Using the measure $q$ we have as follows. From \eqref{eq.4.42} we have the expansion
\begin{equation}\label{eq.4.52}
q[\mathfrak{a}, (k,k)]=\left(1-\frac{2\alpha^*_1}{k}\right)q[\mathfrak{a}, (k,k+1)]+O\left(\frac{1}{k^2}\right),
\end{equation}
and from \eqref{eq.4.41} we have the expansion
\begin{equation}\label{eq.4.53}
\begin{aligned}
q[\mathfrak{a}, (k,k+1)]&=\frac{1}{2}\left(1-\frac{2\alpha^*_{2}}{k}\right)q[\mathfrak{a},(k-1,k+1)]\\
+&\frac{1}{2}q[\mathfrak{a}, (k,k)]+
O\left(\frac{1}{k^2}\right).
\end{aligned}
\end{equation}
From \eqref{eq.4.52} and \eqref{eq.4.53} we obtain
\begin{equation}\label{eq.4.54}
q[\mathfrak{a},(k-1,k+1)]\left(1+\frac{2\alpha_1^*}{k}+\frac{2\alpha_2^*}{k}\right)q[\mathfrak{a},(k,k+1)]+O\left(\frac{1}{k^2}\right).
\end{equation}
So, \eqref{eq.4.70} for $l=1$ follows. The proof of \eqref{eq.4.71} for $l=2$ follows by a similar way using expansion \eqref{eq.4.40}, and we do not present the details.

Assume that it is shown the justice of expansions \eqref{eq.4.70} and \eqref{eq.4.71} for a specific value of $l=l_0$. That is,
\begin{equation}\label{eq.4.72}
\begin{aligned}
q[\mathfrak{a}, (k-l_0,k+l_0)]&=\left(1+\frac{2\alpha^*_{2l_0-1}}{k}+\frac{2\alpha^*_{2l_0}}{k}\right)\\
\times&q[\mathfrak{a}, (k-l_0+1,k+l_0)]+O\left(\frac{1}{k^2}\right),
\end{aligned}
\end{equation}
\begin{equation*}\label{eq.4.73}
\begin{aligned}
q[\mathfrak{a}, (k-l_0+1,k+l_0)]&=\left(1+\frac{2\alpha^*_{2l_0-2}}{k}+\frac{2\alpha^*_{2l_0-1}}{k}\right)\\
\times&q[\mathfrak{a}, (k-l_0+1,k+l_0-1)]+O\left(\frac{1}{k^2}\right).
\end{aligned}
\end{equation*}
Prove these expansions for $l=l_0+1$. We have
\begin{equation}\label{eq.4.74}
\begin{aligned}
&q[\mathfrak{a}, (k-l_0,k+l_0)]\\
&=\frac{1}{2}\left(1-\frac{2\alpha^*_{2l_0+1}}{k}\right)q[\mathfrak{a}, (k-l_0,k+l_0+1)]\\
&\ \ \ +\frac{1}{2}\left(1+\frac{2\alpha^*_{2l_0-1}}{k}\right)q[\mathfrak{a}, (k-l_0+1,k+l_0)]+O\left(\frac{1}{k^2}\right).
\end{aligned}
\end{equation}
Taking into account that \eqref{eq.4.72} can be rewritten in the form
$$
\begin{aligned}
q[\mathfrak{a}, (k-l_0+1,k+l_0)]&=\left(1-\frac{2\alpha^*_{2l_0-1}}{k}-\frac{2\alpha^*_{2l_0}}{k}\right)\\
\times&q[\mathfrak{a}, (k-l_0,k+l_0)]+O\left(\frac{1}{k^2}\right).
\end{aligned}
$$
and substituting it into \eqref{eq.4.74} we obtain the asymptotic expression between the terms $q[\mathfrak{a}, (k-l_0,k+l_0+1)]$ and $q[\mathfrak{a}, (k-l_0,k+l_0)]$, which after algebra can be presented as
$$
\begin{aligned}
q[\mathfrak{a}, (k-l_0,k+l_0+1)]&=\left(1+\frac{2\alpha^*_{2l_0}}{k}+\frac{2\alpha^*_{2l_0+1}}{k}\right)\\
\times&q[\mathfrak{a}, (k-l_0,k+l_0)]+O\left(\frac{1}{k^2}\right).
\end{aligned}
$$
So, \eqref{eq.4.71} is proved. The proof of \eqref{eq.4.70} is similar.
\end{proof}

\begin{cor}\label{C1} As $k\to\infty$, we have the following asymptotic expansions. For $2\leq l\leq k-1$,
\begin{equation}\label{eq.4.50}
\begin{aligned}
&q[\mathfrak{a}, (k-l,k+l)]=
\left(1+\frac{2\alpha^*_{2l-2}}{k}
+\frac{4\alpha^*_{2l-1}}{k}+\frac{2\alpha^*_{2l}}{k}\right)\\
&\ \ \times q[\mathfrak{a}, (k-l+1,k+l-1)]+O\left(\frac{1}{k^2}\right),
\end{aligned}
\end{equation}
and for $1\leq l\leq k-1$
\begin{equation}\label{eq.4.51}
\begin{aligned}
&q[\mathfrak{a},(k-l,k+l+1)]=\left(1+\frac{2\alpha^*_{2l-1}}{k}+\frac{4\alpha^*_{2l}}{k}
+\frac{2\alpha^*_{2l+1}}{k}\right)\\
&\ \ \times q[\mathfrak{a}, (k-l+1,k+l)]+O\left(\frac{1}{k^2}\right).
\end{aligned}
\end{equation}
\end{cor}
\begin{proof}
The proof of Corollary \ref{C1} follows immediately from Lemma \ref{L1}.
\end{proof}

The next corollary provides estimates for $q[\mathfrak{a}, (0,2k)]$ and $q[\mathfrak{a}, (0,2k+1)]$ as $k\to\infty$.
\begin{cor}\label{C2}
As $k\to\infty$, the following estimates hold
\begin{eqnarray}
q[\mathfrak{a}, (0,2k)]&=&q[\mathfrak{a}, (1,2k-1)]+O\left(\frac{1}{k}\right),\label{eq.4.55}\\
q[\mathfrak{a}, (0,2k+1)]&=&q[\mathfrak{a}, (1,2k)]+O\left(\frac{1}{k}\right),\label{eq.4.56}
\end{eqnarray}
\end{cor}

\begin{proof}
Prove \eqref{eq.4.55}. For large $k$, the following expansion follows immediately from \eqref{eq.3.18}
\begin{equation}\label{eq.4.57}
\begin{aligned}
q[\mathfrak{a}, (1,2k)]&=\frac{1}{4}\bigg(q[\mathfrak{a}, (0,2k)]+q[\mathfrak{a}, (2,2k)]\\
+&q[\mathfrak{a}, (1,2k+1)]+q[\mathfrak{a}, (1,2k-1)]\bigg)+O\left(\frac{1}{k}\right).
\end{aligned}
\end{equation}
From Lemma \ref{L1} and Corollary \ref{C1} we have the following estimates:
\begin{eqnarray*}
  q[\mathfrak{a}, (1,2k)] &=& q[\mathfrak{a}, (1,2k-1)]+O\left(\frac{1}{k}\right), \\
  q[\mathfrak{a}, (2,2k)] &=& q[\mathfrak{a}, (1,2k-1)]+O\left(\frac{1}{k}\right), \\
  q[\mathfrak{a}, (1,2k+1)] &=& q[\mathfrak{a}, (1,2k)]+O\left(\frac{1}{k}\right).
\end{eqnarray*}
Substitution of these estimates into \eqref{eq.4.57} and rearranging the terms yields \eqref{eq.4.55}. The proof of \eqref{eq.4.56} is similar.
\end{proof}

The following lemma is an analogue of \cite[Theorem 3]{A3} proved earlier for birth-and-death processes.
Recall that the $K$th iterate of natural logarithm is denoted by $\ln_{(K)}x$, where $\ln_{(1)}x=\ln x$ and for any $ k\geq 2$, $\ln_{(k)}x=\ln_{(k-1)}(\ln x)$.

\begin{lem}\label{L3} Let
\begin{eqnarray*}
\lambda_n(\mathfrak{a})&=&P_0(\mathfrak{a}; n+1|n);\\
\mu_n(\mathfrak{a})&=&P_0(\mathfrak{a}; n-1|n).
\end{eqnarray*}
The random walk $\mathbf{S}_t(\mathfrak{a})$ is transient if there exist $c>1$, $K\geq1$ and a number $n_0$ such that for all $n>n_0$
\begin{equation*}
\frac{\lambda_n(\mathfrak{a})}{\mu_n(\mathfrak{a})}\geq1+\frac{1}{n}+\sum_{k=1}^{K-1}\frac{1}{n\prod_{j=1}^{k}\ln_{(j)}n}+\frac{c}{n\prod_{k=1}^{K}\ln_{(k)}n},
\end{equation*}
and is recurrent if there exists $K\geq1$ and a number $n_0$ such that for all $n>n_0$
\begin{equation*}
\frac{\lambda_n(\mathfrak{a})}{\mu_n(\mathfrak{a})}\leq1+\frac{1}{n}+\sum_{k=1}^{K}\frac{1}{n\prod_{j=1}^{k}\ln_{(j)}n}.
\end{equation*}
\end{lem}
The proof of Lemma \ref{L3} in turn is based on the following lemma.
\begin{lem}\label{L4} The random walk $\mathbf{S}_t(\mathfrak{a})$ is recurrent if and only if
\begin{equation}\label{eq.4.58}
\sum_{k=1}^{\infty}\prod_{j=1}^{k}\frac{\mu_j(\mathfrak{a})}{\lambda_j(\mathfrak{a})}=\infty.
\end{equation}
\end{lem}
\begin{proof}
The parameters $\lambda_n(\mathfrak{a})$ and $\mu_n(\mathfrak{a})$ describe the behavior of the process $\|\mathbf{S}_t\|$ that is not Markov. We will show below that the original process $\|\mathbf{S}_t\|$ behaves similarly as the birth-and-death process with the birth and death rates $\lambda_n(\mathfrak{a})$ and $\mu_n(\mathfrak{a})$ that coincide with the probabilities $P_0(\mathfrak{a}; n+1|n)$ and $P_0(\mathfrak{a}; n-1|n)$, respectively. We prove this lemma by the method of \cite{A}. The method is alternate to the known method of Lamperti \cite{Lamperti1} (see also \cite{M_et_all}).

Consider the series of random walks $\mathbf{S}_t(\mathfrak{a},N)$, $N=1,2,\ldots$, where the parameter $N$ characterizes the maximally possible value of $I_{t}+J_{t}$ (the notation for $I_t$ and $J_t$ is given in Section \ref{S1.2}). Specifically, if $I_{t}+J_t<N$, then the random walk behaves as the original random walk described by \eqref{eq.1.4} -- \eqref{eq.1.13}. However, if $I_{t}+J_{t}=N$, then it is assumed that $I_{t+1}+J_{t+1}=N-1$. More specifically,
\begin{eqnarray*}
\begin{aligned}
\mathsf{P}\{I_{t+1}=I_t, J_{t+1}=J_t-1~|~N_t=N, I_t>0\}
&=&\frac{1}{2}+\frac{2\alpha_{I_t,J_t}}{N},\\
\mathsf{P}\{I_{t+1}=I_t-1, J_{t+1}=J_t~|~N_t=N, I_t>0\}
&=&\frac{1}{2}-\frac{2\alpha_{I_t,J_t}}{N},
\end{aligned}
\end{eqnarray*}
\begin{equation*}
\mathsf{P}\{J_{t+1}=J_t-1~|~N_t=N, I_t=0\}
=1.
\end{equation*}

A random walk with a fixed boundary value $N$ is shown in the figure below. Here in the figure $\alpha_{n}$ stands for $\alpha_{\mathbf{n}}$ with $\|\mathbf{n}\|<N$, and $\alpha_{m}$ stands for $\alpha_{\mathbf{m}}$ with $\|\mathbf{m}\|=N$.
\begin{figure}
\includegraphics[width=10cm, height=14cm]{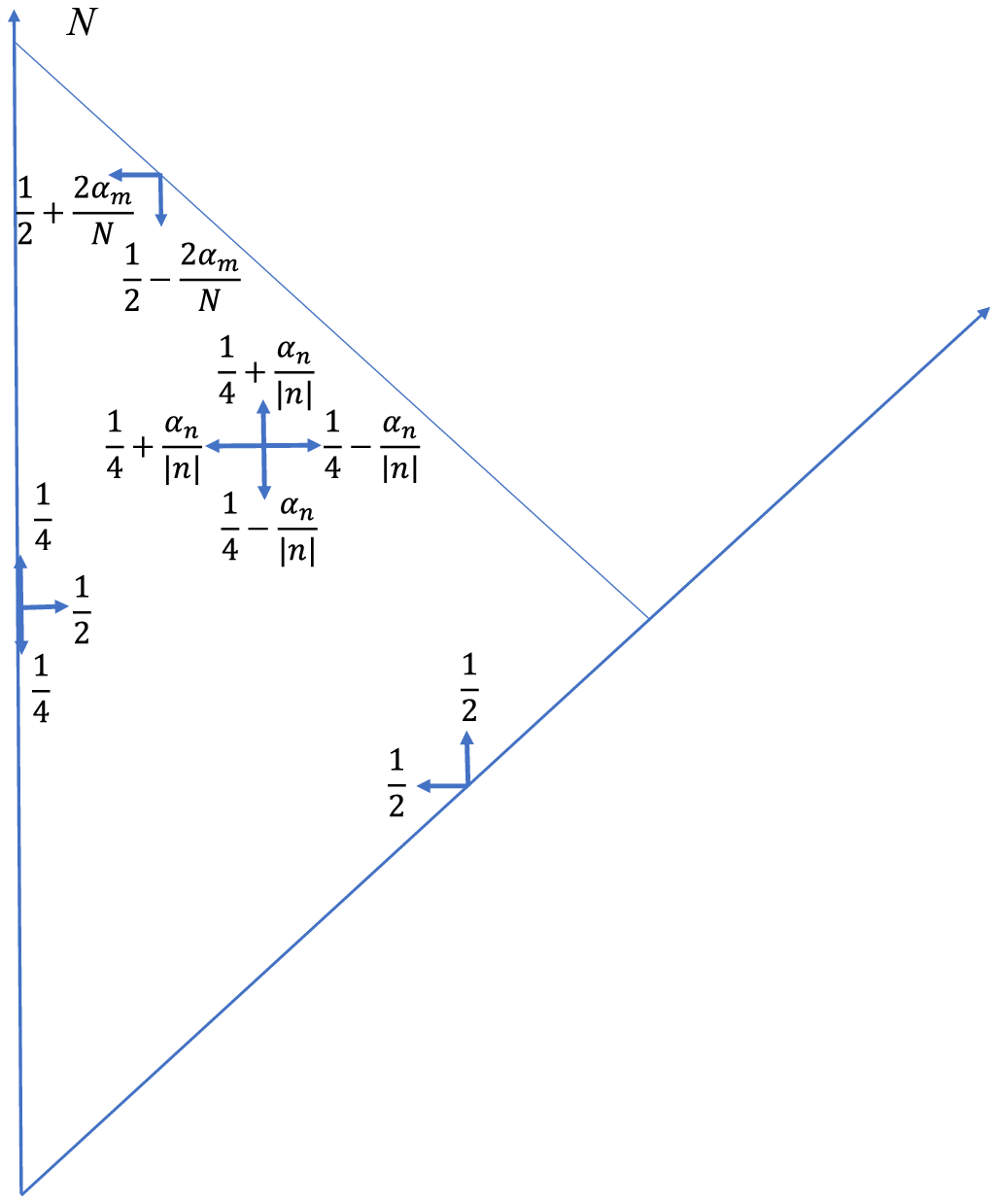}
\caption{Transition probabilities for a random walk in a triangle.}
\end{figure}

With fixed $N$, we derive the system of equations similar to that given in Section \ref{S3} for $P_\tau(\mathfrak{a},\mathbf{n})$, and then define the measure $p(\mathfrak{a},\mathbf{n},N)$ by the formula
\begin{equation}\label{eq.4.60}
p(\mathfrak{a},\mathbf{n},N)=\lim_{\tau\to\infty}\frac{P_\tau(\mathfrak{a},\mathbf{n},N)}{P_\tau(\mathfrak{a},\mathbf{0},N)}.
\end{equation}
Limit relation \eqref{eq.4.60} is similar to that given by \eqref{eq.3.30}. However,
unlike \eqref{eq.3.30}, both the numerator and denominator have the nonzero limits as $\tau\to\infty$, and hence, the proof of the existence of the limit on the right-hand side of \eqref{eq.4.60} does not require an effort.

Since both $P_\tau(\mathfrak{a}, \mathbf{n}, N)$ and $P_\tau(\mathfrak{a}, \mathbf{0}, N)$ are monotone decreasing as $N\to\infty$, we have
\begin{equation}\label{eq.4.4}
p(\mathfrak{a}, \mathbf{n}, \infty)=\lim_{N\to\infty}p(\mathfrak{a}, \mathbf{n}, N)=\lim_{N\to\infty}\lim_{\tau\to\infty}\frac{P_\tau(\mathfrak{a}, \mathbf{n}, N)}{P_\tau(\mathfrak{a}, \mathbf{0}, N)}.
\end{equation}

The last convergence is uniform. Indeed, let $N$ be large and $1<\|\mathbf{n}\|<N$. Then the system of equations for $P_\tau(\mathfrak{a}, \mathbf{n}, N)$ is the same as that given by \eqref{eq.3.1} -- \eqref{eq.3.13} for $P_\tau(\mathfrak{a}, \mathbf{n})$, while $P_\tau(\mathfrak{a}, \mathbf{0}, N)$ is found from the normalization condition and satisfies the inequality $P_\tau(\mathfrak{a}, \mathbf{0}, N)>P_\tau(\mathfrak{a}, \mathbf{0})$. Furthermore, for $N_2>N_1$ we have $P_\tau(\mathfrak{a}, \mathbf{0}, N_2)<P_\tau(\mathfrak{a}, \mathbf{0}, N_1)$. On the other
hand, $P_\tau(\mathfrak{a}, \mathbf{n}, N)$ is derived from the recursion that starts from $P_\tau(\mathfrak{a}, \mathbf{0}, N)$. Hence, in the calculation of $p(\mathfrak{a}, \mathbf{n}, \infty)$ the nonzero quantity $P_\tau(\mathfrak{a}, \mathbf{0}, N)$ that
is presented in the numerator and denominator reduces. So, for $\|\mathbf{n}\|<N_1$ we arrive at the exact equality $p(\mathfrak{a}, \mathbf{n}, N_1)=p(\mathfrak{a}, \mathbf{n}, N_2)$. This enables us to
conclude that the convergence on $N$ in \eqref{eq.4.4} is uniform. Hence the order
of limits in \eqref{eq.4.4} can be interchanged, and
$$
p(\mathfrak{a}, \mathbf{n})=p(\mathfrak{a}, \mathbf{n}, \infty),
$$
where $p(\mathfrak{a}, \mathbf{n})$ is defined by \eqref{eq.3.30}, while $p(\mathfrak{a}, \mathbf{n}, \infty)=\lim_{N\to\infty}p(\mathfrak{a}, \mathbf{n}, N)$.

Denoting
$$
p_j(\mathfrak{a},N)=\sum_{\|\mathbf{n}\|=j}p(\mathfrak{a},\mathbf{n},N), \ j=0,1,\ldots, N,
$$
one can see that $p_j(\mathfrak{a},N)$ is associated with the stationary distribution of $\|\mathbf{S}_t(\mathfrak{a},N)\|$ on the one hand, and with the stationary distribution of the birth-and-death process with specified birth and death rates denoted by $\lambda_j(\mathfrak{a},N)$ and $\mu_j(\mathfrak{a},N)$, in which $\lambda_{N+1}(\mathfrak{a},N)=0$, on the other hand. With $N=1,2,\ldots$ we obtain the sequence $\{p_j(\mathfrak{a},N), j=0,1,\ldots,N\}_{N\geq1}$, where
%
%

$$
\lim_{N\to\infty}p_j(\mathfrak{a},N)=\sum_{\|\mathbf{n}\|=j}p(\mathfrak{a}, \mathbf{n}),
$$
and $\lim_{N\to\infty}\lambda_j(\mathfrak{a},N)=\lambda_j(\mathfrak{a})$ and $\lim_{N\to\infty}\mu_j(\mathfrak{a},N)=\mu_j(\mathfrak{a})$.
The last limit relations imply that the random walk $\mathbf{S}_t(\mathfrak{a})$ is recurrent if and only if the birth-and-death process with birth and death rates $\lambda_j(\mathfrak{a})$ and $\mu_j(\mathfrak{a})$, $j=1,2,\ldots$ is recurrent.
\end{proof}

\noindent
\textit{Proof of Lemma} \ref{L3}. The proof of this lemma follows immediately from Lemma \ref{L4} by following application of the extended Bertrand-De Morgan test for convergence or divergence of positive series \cite{A3}.

\subsection{Proof of Theorem \ref{T1}} To prove the theorem, we are to derive the expression for
$$
\lambda_n(\mathfrak{a})=P_0(\mathfrak{a}; n+1|n),
$$
to find then the limit
$$
\ell(\mathfrak{a})=\lim_{n\to\infty}\left(\frac{\lambda_n(\mathfrak{a})}{1-\lambda_n(\mathfrak{a})}\right)^n=\lim_{n\to\infty}\left(\frac{\lambda_n(\mathfrak{a})}{\mu_n(\mathfrak{a})}\right)^n.
$$

Note, that the number of possible cases when the total number of customers in the network of two queueing systems is equal to $n$ is $n+1$.
The two of these cases correspond to the situation when the number of customers in one of the systems is $n$, and another system is empty. The rest $n-1$ cases correspond to the situation when each of the queues contains at least one customer. Assume that $n$ is large and even. (The assumption that $n$ is even is technical. For odd $n$ the arguments are similar.)
Then, the corresponding states of the system are presented as
\begin{equation}\label{eq.5.2}
(0,2k), (1,2k-1),\ldots,(k-1,k), (k,k).
\end{equation}

The relative frequencies of the states in the order of their appearance in \eqref{eq.5.2} are proportional to
\begin{equation}\label{eq.5.3}
\begin{aligned}
&\frac{1}{2}q[\mathfrak{a}, (0,2k)], q[\mathfrak{a}, (1,2k-1)], \ldots\\
&\ldots, q[\mathfrak{a},(k-1,k+1)], \frac{1}{2}q[\mathfrak{a}, (k,k)].
\end{aligned}
\end{equation}
The first and the last terms in \eqref{eq.5.3} enter with coefficient 1/2, since for the $q$-measures of their states we have the relationship $q[\mathfrak{a}, (0,2k)]=2p[\mathfrak{a}, (0,2k)]$ and $q[\mathfrak{a}, (k,k)]=2p[\mathfrak{a}, (k,k)]$, respectively, while for all other states $(k-l,k+l)$, $1\leq l\leq k-1$, we have $q[\mathfrak{a}, (k-l,k+l)]=p[\mathfrak{a}, (k-l,k+l)]$.

The first and last terms in \eqref{eq.5.3} are marginal terms, the other terms are expressed via recurrence relations \eqref{eq.4.50} given in Corollary \ref{C1}.
Let us set the relative frequency of state $(1,2k-1)$ in \eqref{eq.5.2} to a 2. Then, according to the estimate in Corollary \ref{C2}, the relative frequency of state $(0,2k)$ may be set to equal to $1+c/k$ with some constant $c$, and the relative frequency of state $(2,2k-2)$, according to asymptotic expansion \eqref{eq.4.50}, is \textit{approximately} equal to
$$
2\left(1-\frac{2\alpha^*_{2k-2}}{k}-\frac{4\alpha^*_{2k-3}}{k}-\frac{2\alpha^*_{2k-4}}{k}\right).
$$
Here and later the word \textit{approximately} means that the term $O(1/k^2)$ of the asymptotic expansion is ignored. Then, according to same expansion \eqref{eq.4.50}, the relative frequency of state $(3,2k-3)$ is approximately equal to
$$
2\left(1-\frac{2\alpha^*_{2k-2}}{k}-\frac{4\alpha^*_{2k-3}}{k}-\frac{2\alpha^*_{2k-4}}{k}\right)\left(1-\frac{2\alpha^*_{2k-4}}{k}-\frac{4\alpha^*_{2k-5}}{k}-\frac{2\alpha^*_{2k-6}}{k}\right).
$$
Then, the general formula for the approximate relative frequency of state $(l, 2k-l)$, $2\leq l\leq k-1$ is
$$
2\prod_{i=1}^{l-1}\left(1-\frac{2\alpha^*_{2k-2i}}{k}-\frac{4\alpha^*_{2k-2i-1}}{k}-\frac{2\alpha^*_{2k-2i-2}}{k}\right).
$$
So, the total approximate relative frequency of the states in \eqref{eq.5.2} is
\begin{equation}\label{eq.5.4}
\begin{aligned}
&\underbrace{\left(1+\frac{c}{k}\right)}_{\text{rel. frequency of} \ (0,2k)}+2\\
&+2\sum_{j=1}^{k-1}\prod_{i=1}^{j}\left(1-\frac{2\alpha^*_{2k-2i}}{k}-\frac{4\alpha^*_{2k-2i-1}}{k}-\frac{2\alpha^*_{2k-2i-2}}{k}\right)\\
&+\underbrace{\frac{q[\mathfrak{a}, (k,k)]}{q[\mathfrak{a}, (1,2k-1)]}}_{\text{rel. frequency of} \ (k,k)}.
\end{aligned}
\end{equation}
Here in \eqref{eq.5.4} the sum of the first two and last term
$$
\left(1+\frac{c}{k}\right)+2+\frac{q[\mathfrak{a}, (k,k)]}{q[\mathfrak{a}, (1,2k-1)]}
$$
is an exact value, while the sum of the rest terms
\begin{equation}\label{eq.5.1}
2\sum_{j=1}^{k-1}\prod_{i=1}^{j}\left(1-\frac{2\alpha^*_{2k-2i}}{k}-\frac{4\alpha^*_{2k-2i-1}}{k}-\frac{2\alpha^*_{2k-2i-2}}{k}\right)
\end{equation}
has the accuracy $O(1/k^2)$.

Note that following \eqref{eq.1.4} -- \eqref{eq.1.7} and \eqref{eq.1.8} for all $t>1$, we obtain
\begin{equation}\label{eq.1.15}
\begin{aligned}
&\mathsf{P}\{\|\mathbf{S}_t(\mathfrak{a})\|=n+1~\big|~\|\mathbf{S}_{t-1}(\mathfrak{a})\|=n, \mathcal{R}_{t-1}\}\\
&=\mathsf{P}\{\|\mathbf{S}_t(\mathfrak{a})\|=n-1~\big|~\|\mathbf{S}_{t-1}(\mathfrak{a})\|=n, \mathcal{R}_{t-1}\}=\frac{1}{2},
\end{aligned}
\end{equation}
where $\mathcal{R}_{t}=\{I_{t}>0\}$. (The notation for $I_t$  is introduced in Section \ref{S1}.)

Then the terms $P_n(\mathfrak{a})$ and $Q_n(\mathfrak{a})$  for large $n$ are evaluated as follows. If both of queues are not empty, then according to \eqref{eq.1.15}, the probability of incrementing of the total number of customers is equal to the probability of its decrementing. Specifically, the possible number of incrementing or decrementing given $\mathcal{R}_{t-1}$, which are equal because of the symmetry, are measured by
\begin{equation}\label{eq.5.5}
\begin{aligned}
&2\left[1+\sum_{j=1}^{k-1}\prod_{i=1}^{j}\left(1-\frac{2\alpha^*_{2k-2i}}{k}-\frac{4\alpha^*_{2k-2i-1}}{k}-\frac{2\alpha^*_{2k-2i-2}}{k}\right)\right]\\
&+\frac{q[\mathfrak{a}, (k,k)]}{q[\mathfrak{a}, (1,2k-1)]}.
\end{aligned}
\end{equation}
Here and later, the word \textit{measured} means that we exclude the asymptotically small term $O(1/k^2)$
Expression \eqref{eq.5.5} is obtained by eliminating the term for relative frequency of $(0,2k)$.

If one of the queues is empty, then incrementing  occurs in the case of arrival of a customer in one of two queues, while decrementing occurs at the moment of a service completion in the queue containing a customer in service. In this case, the number of additionally possible incrementing is measured by
\begin{equation}\label{eq.5.6}
\left(1+\frac{1}{2}\right)\left(1+\frac{c}{k}\right)=\frac{3}{2}\left(1+\frac{c}{k}\right),
\end{equation}
while the number of additionally possible decrementing is measured by
\begin{equation}\label{eq.5.7}
\frac{1}{2}\left(1+\frac{c}{k}\right).
\end{equation}
Recall that an arrival rate when the system is empty is 2. This explains the presence of the coefficient (1+1/2) rather than 1 in the first brackets on the left-hand side of \eqref{eq.5.6}, while for the number of additionally possible decrementing the corresponding coefficient is 1/2. This makes the proportion between the numbers of additionally possible incrementing to additionally possible decrementing given by \eqref{eq.5.6} and \eqref{eq.5.7}, respectively, to equal to 3 (that is greater than 1), which is important for the following application of Lemma \ref{L3}.

So, based on these facts, for $\lambda_n(\mathfrak{a})=\lambda_{2k}(\mathfrak{a})$ we have the following presentation:
$$
\lambda_{2k}(\mathfrak{a})=\frac{A_k}{B_k},
$$
where
$$
\begin{aligned}
A_k=&\frac{3}{2}\left(1+\frac{c}{k}\right)+2\\
&+2\sum_{j=1}^{k-1}\prod_{i=1}^{j}\left(1-\frac{2\alpha^*_{2k-2i}}{k}-\frac{4\alpha^*_{2k-2i-1}}{k}-\frac{2\alpha^*_{2k-2i-2}}{k}\right)\\
&+\frac{q[\mathfrak{a}, (k,k)]}{q[\mathfrak{a}, (1,2k-1)]},
\end{aligned}
$$
and
$$
\begin{aligned}
B_k=&2\left(1+\frac{c}{k}\right)+4\\
&+4\sum_{j=1}^{k-1}\prod_{i=1}^{j}\left(1-\frac{2\alpha^*_{2k-2i}}{k}-\frac{4\alpha^*_{2k-2i-1}}{k}-\frac{2\alpha^*_{2k-2i-2}}{k}\right)\\
&+2\frac{q[\mathfrak{a}, (k,k)]}{q[\mathfrak{a}, (1,2k-1)]}.
\end{aligned}
$$
So,
\begin{equation}\label{eq.5.8}
\ell(\mathfrak{a})=\lim_{k\to\infty}\left(\frac{P_{2k}(\mathfrak{a})}{1-P_{2k}(\mathfrak{a})}\right)^{2k}=\lim_{k\to\infty}\left(\frac{A_k}{B_k-A_k}\right)^{2k}.
\end{equation}
From the definition of $\kappa(a)$ given in \eqref{eq.1.22} for large $k$, we obtain
$$
\sum_{j=1}^{k-1}\prod_{i=1}^{j}\left(1-\frac{2\alpha^*_{2k-2i}}{k}-\frac{4\alpha^*_{2k-2i-1}}{k}-\frac{2\alpha^*_{2k-2i-2}}{k}\right)\asymp \frac{k}{\kappa(\mathfrak{a})}.
$$
Substituting this expansion into \eqref{eq.5.8}, we obtain
\begin{equation}\label{eq.5.9}
\begin{aligned}
\ell(\mathfrak{a})&=\lim_{k\to\infty}\left(\frac{1+\frac{q[\mathfrak{a}, (k,k)]}{q[\mathfrak{a}, (1,2k-1)]}+\frac{3}{2}+\frac{3c}{2k}+\frac{2k}{\kappa(\mathfrak{a})}}{1+\frac{q[\mathfrak{a}, (k,k)]}{q[\mathfrak{a}, (1,2k-1)]}+\frac{1}{2}+\frac{c}{2k}+\frac{2k}{\kappa(\mathfrak{a})}}\right)^{2k}\\
&=\lim_{k\to\infty}\left(\frac{\frac{\kappa(\mathfrak{a})}{2k}+\frac{\kappa(\mathfrak{a})q[\mathfrak{a}, (k,k)]}{2q[\mathfrak{a}, (1,2k-1)]k}+\frac{3\kappa(\mathfrak{a})}{4k}+\frac{3c\kappa(\mathfrak{a})}{4k^2}+1}{\frac{\kappa(\mathfrak{a})}{2k}+\frac{\kappa(\mathfrak{a})q[\mathfrak{a}, (k,k)]}{2q[\mathfrak{a}, (1,2k-1)]k}+\frac{\kappa(\mathfrak{a})}{4k}+\frac{3c\kappa(\mathfrak{a})}{4k^2}+1}\right)^{2k}\\
&=\frac{\exp\left(\kappa(\mathfrak{a})+\kappa(\mathfrak{a})\lim_{k\to\infty}\frac{q[\mathfrak{a}, (k,k)]}{2q[\mathfrak{a}, (1,2k-1)]}+\frac{3\kappa(\mathfrak{a})}{2}\right)}{\exp\left(\kappa(\mathfrak{a})+\kappa(\mathfrak{a})\lim_{k\to\infty}\frac{q[\mathfrak{a}, (k,k)]}{2q[\mathfrak{a}, (1,2k-1)]}+\frac{\kappa(\mathfrak{a})}{2}\right)}\\
&=\mathrm{e}^{\kappa(\mathfrak{a})}.
\end{aligned}
\end{equation}

 It follows from Lemma \ref{L3} that if
$\kappa(a)\leq 1$, then the random walk is recurrent. Otherwise, it is transient. In the case of $\kappa(\mathfrak{a})=1$ the required expansion of $\lambda_n(\mathfrak{a})/\mu_n(\mathfrak{a})$ is
$$
\frac{\lambda_{2k}(\mathfrak{a})}{\mu_{2k}(\mathfrak{a})}=1+\frac{1}{2k}+O\left(\frac{1}{k^2}\right),
$$
where the remainder $O(1/k^2)$ in the right-hand side of the relation is associated with the terms $3c\kappa(\mathfrak{a})/(4k^2)$ and $c\kappa(\mathfrak{a})/(4k^2)$ in the numerator and denominator, respectively, of the fraction in the second line of \eqref{eq.5.9} and with the remainder of expansions in \eqref{eq.5.1} that is also $O(1/k^2)$.
The theorem is proved.

\section{Discussion and future research}\label{S5}
\subsection{Discussion} The main result of the paper is based on presentation of $\kappa(\mathfrak{a})$ in the form of \eqref{eq.1.22} or \eqref{eq.1.50}. If the sequence of products \eqref{eq.1.23} converges, then the basic term in this presentation is the aforementioned sequence of products.
When $k$ is large, we have
$$
\begin{aligned}
\kappa(\mathfrak{a})&=\prod_{i=1}^{k-1}\left(1+\frac{2\alpha^*_{2k-2i}}{k}+\frac{4\alpha^*_{2k-2i-1}}{k}+\frac{2\alpha^*_{2k-2i-2}}{k}\right)\\
&\approx\prod_{i=0}^{2k}\left(1+\frac{4\alpha^*_{2k-i}}{k}\right).
\end{aligned}
$$
The last presentation enables us to conclude that, as $k$ taken large enough, the same asymptotic result holds for all permutations of
$$
\alpha^*_0, \alpha^*_1, \ldots, \alpha^*_{2k}.
$$
That is, replacing $\alpha^*_m$ with $\alpha^*_{j_m}$ $(m=1,2,\ldots, 2k)$, where $j_1$, $j_2$,\ldots, $j_{2k}$ is the permutation of the indices 1, 2,\ldots, $2k$, we arrive at the same classification as that for the original random walk.

\subsection{Future research}
In the present paper we provided a complete study of a new family of random walks. The interesting extension of the present study seems to be in the case where the elements $\mathfrak{a}$ of the set $\mathcal{A}$ are infinite-dimensional \textit{random} vectors. That is, $\alpha_{\mathbf{n}}$ are random variables taking the values in bounded areas that are defined by condition \eqref{eq.1.20}. Then, assumption \eqref{eq.1.21} should be replaced by
$$
|\mathsf{E}\alpha_{\mathbf{n}+\mathbf{1}}-\mathsf{E}\alpha_{\mathbf{n}}|\leq\gamma|\mathsf{E}\alpha_{\mathbf{n}}-\mathsf{E}\alpha_{\mathbf{n}-\mathbf{1}}|, \quad 0<\gamma<1,
$$
or another similar assumption.

Another possible direction of the extension of the main result of the present paper is to study random walks beyond nearest-neighborhood case.

\section*{Disclosure of the potential conflict of interests}
The author declares that he has no conflict of interests.

%
%

\section*{Acknowledgement}
The author expresses his gratitude to all the people who made critical comments officially or privately.


\begin{thebibliography}{10}
\bibitem{A} \textsc{Abramov, V. M.} (2018). Conservative and semiconservative random walks: recurrence and transience. \emph{J. Theoret. Probab.} \textbf{31} (3), 1900--1922. Correction. \emph{J. Theoret. Probab.} \textbf{32} (2019) (1), 541--543.  
\bibitem{A3}\textsc{Abramov, V. M.} (2020). Extension of the Bertrand-De Morgan test and its application. \textit{Amer. Math. Monthly}, \textbf{127} (5), 444--448.
\bibitem{A4}\textsc{Abramov, V. M.} (2022). Necessary and sufficient conditions for the convergence of positive series. \textit{J. Classical Anal.} \textbf{19}, 117--125.
\bibitem{A5}\textsc{Abramov,\,V.\,M.} (2023). Conditions for recurrence and transience for time-inhomogeneous birth-and-death processes. \textit{Bull.  Aust. Math. Soc. } https://doi.org/10.1017/S0004972723000539
\bibitem{A6}\textsc{Abramov,\,V.\,M.} (2023). Crossings states and sets of states in random walks. \emph{Methodol. Comp. Appl. Probab.} \textbf{25}, Article No. 28.
\bibitem{AFM} \textsc{Asymont, I. M., Fayolle, G. and Menshikov, M. V.} (1995). Random walks in quarter plane with zero drift: transience and recurrence. \emph{J. Appl. Probab.} \textbf{32} (4), 941--955. 


\bibitem{Chung} \textsc{Chung, K.L.} (1967). \emph{Markov Chains with Stationary Transition Probabilities}, Second edition. Die Grundlehren der mathematischen Wissenschaften, Band 104, Springer-Verlag New York, Inc., New York. 
\bibitem{FIM} \textsc{Fayolle, G., Iasnogorodski, R. and Malyshev, V.} (1999). \emph{Random Walks in a Quarter Plane}. Springer, New York. 
\bibitem{FMM} \textsc{Fayolle, G., Malyshev, V. and Menshikov, M.V.} (1995). \emph{Topics in the Constructive Theory of Countable Markov Chains}. Cambridge Univ. Press, Cambridge. 
\bibitem{F} \textsc{Flatto, L.} (1958). A problem on random walk. \emph{Quart. J. Math.} Oxford Ser. 2, \textbf{9}, 299--300.  
\bibitem{G} \textsc{Gillis, G.} (1956). Centrally biased discrete random walks. \emph{Quart. J. Math.} Oxford Ser. 2, \textbf{7}, 144--152.  
\bibitem{Harris1} \textsc{Harris, T.E.} (1952). First passage and recurrence distributions, \emph{Trans. Amer. Math. Soc.} \textbf{73}, 471--486.  
\bibitem{Harris2} \textsc{Harris, T.E.} (1956). The existence of stationary measures for certain Markov processes. In: \emph{Proc. Third Berkeley Symposium on Mathematical Statistics and Probability}, Berkeley, Volume 2, University of Calif. Press, pp. 113-124. 
\bibitem{H} \textsc{Hodges, J.L.,Jr. and Rosenblatt, M.} (1953). Recurrence time moments in random walks. \emph{Pacific J. Math.} \textbf{3}, 127--136. 
\bibitem{KM}
\textsc{Karlin, S., McGregor, J.} (1957). The classification of the birth-and-death processes. \emph{Trans. Amer. Math. Soc.}
\textbf{86} (2), 366–-400. 
\bibitem{K} \textsc{Klebaner, F.C.} (1989). Stochastic difference equations and generalized Gamma distributions. \emph{Ann. Probab.} \textbf{17} (1), 178--188. 
\bibitem{Kolmogorov}\textsc{Kolmogorov, A.N.} (1950). \textit{Foundations of Probability}. New York, Chelsea Publishing Company. (English translation from: Kolmogoroff, A. N. \emph{Grundbegriffe der Wahrscheinlichkeitrechnung}. Ergebnisse Der Mathematik, 1933.) 
\bibitem{Lamperti1} \textsc{Lamperti, J.} (1960). Criteria for the recurrence or transience of stochastic processes. 1. \emph{J. Math. Anal. Appl.} \textbf{1}, 314--330. 
\bibitem{Lamperti2} \textsc{Lamperti, J.} (1963). Criteria for stochastic processes. 2. Passage time moments. \emph{J. Math. Anal. Appl.} \textbf{7} (1), 127--145. 
\bibitem{MP} \textsc{Mandelbaum, A. and Pats, G.} (1998). State-dependent stochastic networks. Part I: Approximations and applications with continuous diffusion limits. \emph{Ann. Appl. Probab.} \textbf{8} (2), 569--646.  
\bibitem{M} \textsc{Martin, M.} (1941). A sequence of limit tests for the convergence of series. \emph{Bull.
Amer. Math. Soc.} \textbf{47} (6), 452-457.
\bibitem{MAI} \textsc{Menshikov, M.V., Asymont, I.M. and Iasnogorodskii, R.} (1995). Markov processes with
asymptotically zero drifts. \textit{Probl. Informat. Transmis.}, \textbf{31}, 248--261. Translated from \textit{Problemy Peredachi Informatsii}, \textbf{31}, 60--75 (in Russian). 
\bibitem{M_et_all} \textsc{Menshikov, M.V., Popov, S. and Wade A.} (2016). \emph{Non-homogeneous Random Walks: Lyapunov function methods for near critical stochastic systems}. Cambridge University Press, Cambridge. 
\bibitem{Raschel} \textsc{Raschel, K.} (2014). Random walks in a quarter plane, harmonic functions and conformal mappings. With an appendix by Sandro Franceschi. \emph{Stoch. Proc. Appl.} \textbf{124} (10), 3147--3178. 
\bibitem{SS} \textsc{Shneer, S. and Stolyar, A.} (2020). Stability and moment bounds under utility-maximizing service allocations, with application to some infinite networks. \emph{Adv. Appl. Probab.} \textbf{52} (2), 463--490.
\end{thebibliography}
\end{document}